\documentclass[a4paper,11pt]{amsart}

\usepackage{verbatim}
\usepackage{amsmath}
\usepackage{amscd}
\usepackage{amssymb}
\usepackage{amsthm}
\usepackage{graphicx}
\usepackage[all]{xy}

\newtheorem{thm}{Theorem}[section]
\newtheorem{prop}[thm]{Proposition}
\newtheorem{cor}[thm]{Corollary}
\newtheorem{lem}[thm]{Lemma}

\newtheorem{theorem}[thm]{Theorem}
\newtheorem{proposition}[thm]{Proposition}

\newtheorem{corollary}[thm]{Corollary}
\newtheorem{step}{Step}

\theoremstyle{definition}
\newtheorem{df}[thm]{Definition}

\newtheorem{rem}[thm]{Remark}

\newtheorem{expl}[thm]{Example}

\numberwithin{equation}{section}
\numberwithin{figure}{section}
\numberwithin{table}{section}

\newcommand{\bC}{{\mathbb{C}}}
\newcommand{\bP}{{\mathbb{P}}}
\newcommand{\bQ}{{\mathbb{Q}}}
\newcommand{\bR}{{\mathbb{R}}}
\newcommand{\bZ}{{\mathbb{Z}}}
\newcommand{\scM}{{\mathcal{M}}}
\newcommand{\frakm}{\mathfrak{m}}
\newcommand{\frakX}{\mathfrak{X}}
\newcommand{\XG}{X_{\Gamma}}
\newcommand{\YG}{Y_{\Gamma}}
\newcommand{\DeltaG}{\Delta_{\Gamma}}
\newcommand{\omegatilde}{\widetilde{\omega}}
\newcommand{\Phitilde}{\widetilde{\Phi}}
\newcommand{\Grtilde}{\widetilde{\Gr}}
\newcommand{\bCx}{\bC^*}

\newcommand{\Gr}{\operatorname{Gr}}
\newcommand{\GIT}{/\!\!/}
\newcommand{\bigGIT}{\big/\!\!\big/}
\newcommand{\GITl}{\backslash \!\! \backslash}
\newcommand{\bld}{\boldsymbol}
\newcommand{\Euc}{\mathbb{R}^3}

\newcommand{\pol}{\mathcal{M}_{\bld r}}

\newcommand{\su}{\mathfrak{su}(2)}

\newcommand{\tr}{\operatorname{tr}}
\newcommand{\Sing}{\operatorname{Sing}}
\newcommand{\Spec}{\operatorname{Spec}}
\newcommand{\Int}{\operatorname{Int}}
\newcommand{\PD}{\operatorname{PD}}
\newcommand{\po}{\mathfrak{PO}}

\newcommand{\dge}{\rotatebox[origin=c]{45}{$\ge$}}
\newcommand{\uge}{\rotatebox[origin=c]{315}{$\ge$}}

\newcommand{\dndots}{\hbox to1.65em{\rotatebox[origin=c]{315}{$\cdots$}}}

\newcommand{\llmd}[2]{\hbox to1.65em{$ \lambda_{#2}^{(#1)}$}}

\newcommand{\br}{{\boldsymbol{r}}}
\newcommand{\bx}{{\boldsymbol{x}}}

\newcommand{\frakg}{{\mathfrak{g}}}
\newcommand{\fraks}{{\mathfrak{s}}}

\newcommand{\fraku}{{\mathfrak{u}}}
\newcommand{\fraksu}{{\mathfrak{su}}}

\newcommand{\PGL}{PGL}

\newcommand{\xto}{\xrightarrow}
\newcommand{\simto}{\xto{\sim}}
\newcommand{\vin}{\rotatebox{90}{$\in$}}

\newcommand{\lb}{\left(}
\newcommand{\rb}{\right)}

\newcommand{\N}{{2n-4}}

\title[Toric degenerations of integrable systems 
       on Grassmannians]
      {Toric degenerations of integrable systems 
       on Grassmannians and polygon spaces}
\author{Yuichi Nohara and Kazushi Ueda}
\date{}
\begin{document}
\maketitle


\begin{abstract}
We introduce a completely integrable system
on the Grassmannian of 2-planes
in an $n$-space
associated with any triangulation
of a polygon with $n$ sides,
and compute the potential function
for its Lagrangian torus fiber.
The moment polytopes of this system
for different triangulations
are related by an integral piecewise-linear transformation,
and the corresponding potential functions are related
by its geometric lift
in the sense of Berenstein and Zelevinsky
\cite{Berenstein-Zelevinsky_TPM}.
\end{abstract}

\section{Introduction}
 \label{sc:introduction}

Let $n$ be an integer greater than two and
$
 \br = (r_1, \dots, r_n)
$
be a sequence of positive real numbers
satisfying
\begin{equation}  \label{eq:non-empty}
  r_i < r_1 + \dots + r_{i-1} + r_{i+1} + \dots + r_n
\end{equation}
for each $i = 1, \dots, n$.
The {\em polygon space} $\pol$ is defined by
\[
  \pol
  = \Bigl. 
    \Bigl\{ \bx = (x_1, \dots, x_n) \in \prod_{i=1}^n S^2(r_i) \, 
     \Bigm| \,
        x_1 + \dots + x_n = 0 \, \Bigr\}
    \Bigr/ SO(3),
\]
where $S^2(r_i) \subset \Euc$ is the two-sphere of radius $r_i$
centered at the origin,
and $SO(3)$ acts diagonally on $\prod_{i=1}^n S^2(r_i)$.
A point $\bx \in \pol$ is regarded
as a congruence class of a closed spatial polygon
with sides $x_1, \dots, x_n$.
The condition \eqref{eq:non-empty} implies
$\pol \ne \emptyset$ and $\dim_{\mathbb{R}} \pol = 2(n-3)$.
The variety $\pol$ is singular if and only if 
$
 \sum_{i=1}^n \epsilon_i r_i = 0
$
for some
$
 (\epsilon_1, \dots, \epsilon_n) \in \{ \pm 1 \}^n, 
$
and the singular locus of $\pol$ consists of
$\bx = (x_1, \dots, x_n) \in \pol$
satisfying
$
 \sum_{i=1}^n \epsilon_i x_i = 0.
$
In particular,
the polygon space $\pol$
has at worst isolated singularities
for any $\br$.
The polygon space $\pol$ has a natural structure
of a projective (and hence K\"{a}hler) variety
coming from the identification
$$
 \pol \cong (\bC \bP^1)^n \GIT \PGL(2, \bC).
$$
The GIT quotient
on the right hand side is a compactification
of the configuration space of $n$ points on $\bC \bP^1$,
which has a long history of investigation
going back to the 19th century.
Furthermore, the Gelfand-MacPherson correspondence
\cite{Gelfand-MacPherson_GGGD}
gives an isomorphism between the polygon space 
and the symplectic reduction of the Grassmannian $\Gr(2,n)$ 
of two-planes in $\bC^n$ by a maximal torus  $T_{U(n)}$ of $U(n)$:
\begin{equation} \label{eq:symp_red1}
 \pol \cong \Gr(2,n) \GIT T_{U(n)}.
\end{equation}
It is also known
\cite{Foth_MSP,
Howard-Manon-Millson,
Jeffrey_EMS,
Millson-Poritz,
Treloar}
that $\pol$ is symplectomorphic to 
the moduli space of parabolic $SU(2)$-bundles on $\mathbb{CP}^1$
for sufficiently small $\br$.

Recall that
a completely integrable system
on a symplectic manifold $(X, \omega)$
of dimension $2N$ is an $N$-tuple of functions 
$$
  \Phi = (\varphi_1, \dots, \varphi_N) : X \longrightarrow \mathbb{R}^N
$$
which are functionally independent
(i.e., $d\varphi_1, \dots, d\varphi_N$ are linearly independent)
on an open dense subset and
mutually Poisson commutative;
$\{ \varphi_i, \varphi_j \} = 0$
for any $i, j$.
The Arnold-Liouville theorem states that generic fibers of $\Phi$ are
Lagrangian tori if the fibers are compact and connected.
The toric moment map on a toric manifold
is an example of a completely integrable system.

Fix a convex planar polygon $P$
with $n$ sides
called the {\em reference polygon}.
The set of triangulations of $P$ is
in natural one-to-one correspondence
with the set of trivalent trees
with $n$ leaves
by sending a triangulation
to its dual graph $\Gamma$.
We often say a ``triangulation'' $\Gamma$
by abuse of notation.
The set of triangulations $\Gamma$
can naturally be identified
with the set of vertices
of the {\em Stasheff associahedron},
which in turn is identified
with the set of ways
to parenthesize a product of $n-1$ elements
into binomial operations.

For any triangulation $\Gamma$
of the reference polygon,
Kapovich and Millson
\cite{Kapovich-Millson_SGPES} and 
Klyachko \cite{Klyachko_SP}
constructed a completely integrable system
\[
  \Phi_{\Gamma} : \pol \to \mathbb{R}^{n-3}
\]
called the {\it bending system}.
For a particular triangulation $\Gamma$
called the {\em caterpillar}
(see Figure \ref{fg:caterpillar}),
the bending system comes from
the {\em Gelfand-Cetlin system}
\cite{Guillemin-Sternberg_GCS}
on the Grassmannian $\Gr(2,n)$ 
through the symplectic reduction \eqref{eq:symp_red1}
\cite{Hausmann-Knutson_PSG}.

The first main result in this paper
is the following:

\begin{thm}\label{th:nohara_system}
 For any triangulation $\Gamma$ of the reference polygon,
 there exists a completely integrable system
 $$
  \Psi_\Gamma : \Gr(2,n) \to \bR^\N
 $$
 which induces the bending system
 $\Phi_{\Gamma}$ on $\pol$
 through the symplectic reduction \eqref{eq:symp_red1}.
\end{thm}


Triangulations of the reference polygon are
related to toric degenerations of $\Gr(2,n)$
by Speyer and Sturmfels \cite{Speyer-Sturmfels_TG}.
On the other hand,
the notion of a {\em toric degeneration of an integrable system}
is introduced in \cite{Nishinou-Nohara-Ueda_TDGCSPF}
(see Definition \ref{def:toric_degeneration}).
The second main result in this paper is the following:

\begin{thm}\label{th:toricdeg_nohara}
 For any triangulation $\Gamma$ of the reference polygon,
 the completely integrable system $\Psi_{\Gamma}$ on $\Gr(2,n)$
 admits a toric degeneration. 
\end{thm}

Since the toric degeneration of $\Gr(2,n)$ is $T_{U(n)}$-invariant
for any triangulation $\Gamma$,
it induces a toric degeneration of $\pol$.
The deformation of $\Psi_{\Gamma}$
in Theorem \ref{th:toricdeg_nohara}
is $T_{U(n)}$-invariant,
and we obtain the following:

\begin{cor}\label{cr:toricdeg_pol}
 The toric degeneration of
 $\Psi_{\Gamma} : \Gr(2,n) \to \mathbb{R}^\N$
 in Theorem $\ref{th:toricdeg_nohara}$ induces a toric degeneration
 of the bending system $\Phi_\Gamma : \pol \to \bR^{n-3}$.
\end{cor}

Kamiyama and Yoshida \cite{Kamiyama-Yoshida} and
Howard, Manon, and Millson \cite{Howard-Manon-Millson} also
studies the relation between toric degenerations of polygon spaces and 
bending systems from a slightly different point of view.

For a triangulation $\Gamma$
of the reference polygon,
let $X_\Gamma$ be the toric variety
obtained as the central fiber
of the corresponding toric degeneration of $\Gr(2,n)$.
The image $\Delta_\Gamma$
of the toric moment map of $X_\Gamma$ coincides
with the moment polytope
of the integrable system
$
 \Psi_\Gamma : \Gr(2,n) \to \bR^\N
$
by Theorem \ref{th:toricdeg_nohara}.

\begin{thm} \label{th:piecewise-linear}
For any pair $(\Gamma, \Gamma')$
of triangulations of the reference polygon,
there is a piecewise-linear automorphism
$$
 T_{\Gamma, \Gamma'} : \bR^\N \to \bR^\N
$$
of the affine space
such that
$
 T_{\Gamma, \Gamma'}(\Delta_{\Gamma'}) = \Delta_\Gamma.
$
The map $T_{\Gamma, \Gamma'}$ is defined over $\bZ$ if 
$\Delta_\Gamma$ is an integral polytope.
\end{thm}

If two triangulations are related
by a single Whitehead move
(see Figure \ref{fig:whitehead1}),
then the corresponding integral piecewise-linear transformation
has the form
$
 (\ldots, u, u_1, u_2, u_3, u_4, \ldots)
  \mapsto (\ldots, u', u_1, u_2, u_3, u_4, \ldots)
$
where
\begin{equation}
 \begin{split}
   u' & = u - \min(u_1+u_2, u_3+u_4) + \min(u_1+u_4, u_2+u_3) \\ 
       & = u + \min( u_1-u_2, u_2-u_1, u_3-u_4, u_4-u_3 )  \\
       & \qquad - \min( u_1-u_4, u_4-u_1, u_2-u_3, u_3-u_2 ).
 \end{split} \label{eq:PL_transf}
\end{equation}
In general,
any two triangulations are related by
a sequence of Whitehead moves,
and the corresponding integral
piecewise-linear transformation
is an iteration
of the transformation above.

The {\em potential function} is a Floer-theoretic invariant
of Lagrangian submanifolds
introduced by Fukaya, Oh, Ohta and Ono
\cite{Fukaya-Oh-Ohta-Ono}.
It gives the ``superpotential''
of the mirror Landau-Ginzburg model
for toric manifolds
\cite{Cho-Oh, FOOO_toric_I, FOOO_LFTMS}.
In \cite{Nishinou-Nohara-Ueda_TDGCSPF},
the potential function
for a Lagrangian torus fiber
of the Gelfand-Cetlin system
\cite{Guillemin-Sternberg_GCS}
is computed
by using a toric degeneration.
An essential point in the argument is the fact
that the central fiber is a toric Fano variety
admitting a small resolution,
which holds also in the present situation:

\begin{thm} \label{th:X_Gamma}
For any triangulation $\Gamma$ of the reference polygon,
the toric variety $X_\Gamma$ is a Fano variety
admitting a small resolution.
\end{thm}

As a result,
we obtain an explicit description of the potential function
for a Lagrangian torus fiber
of the integrable system $\Psi_\Gamma$
as a Laurent polynomial
over the Novikov ring:

\begin{thm} \label{th:potential_function}
The potential function is given by
\begin{equation} \label{eq:potential_function}
  \po_\Gamma = \sum_{\text{triangles}} 
        \left( 
          \frac{y(b) y(c)}{y(a)}
        + \frac{y(a) y(c)}{y(b)}
        + \frac{y(a) y(b)}{y(c)}
        \right),
\end{equation}
where $y(a)$ is a Laurent monomial
defined in Section $\ref{sc:potential_function}$
associated with an edge $a$ of a triangle,
and the sum is taken over all triangles 
in the triangulation $\Gamma$.
\end{thm}


Recall that the {\em geometric lift}
\cite{Berenstein-Zelevinsky_TPM}
of a piecewise-linear function is given
by replacing summation,
subtraction, and the minimum by
multiplication, division, and summation.


\begin{thm} \label{th:geometric_lift}
For any pair $(\Gamma, \Gamma')$
of triangulations of the reference polygon,
the potential functions $\po_\Gamma$ and $\po_{\Gamma'}$
are related by a subtraction-free rational change of variables
obtained as the geometric lift
of the piecewise-linear transformation $T_{\Gamma, \Gamma'}$
in Theorem $\ref{th:piecewise-linear}$.
\end{thm}

If $\Gamma$ and $\Gamma'$ are related
by a single Whitehead move,
then the corresponding change of variables is given by
\begin{equation}
 \begin{split}
  y' = y \cdot
   \frac{y_1 y_4 + y_2 y_3}
        {y_1 y_2 + y_3 y_4}
      = y \cdot
   \frac{\dfrac{y_1}{y_2} + \dfrac{y_2}{y_1} 
     + \dfrac{y_3}{y_4} + \dfrac{y_4}{y_3}}
   {\dfrac{y_1}{y_4} + \dfrac{y_4}{y_1}
     + \dfrac{y_2}{y_3} + \dfrac{y_3}{y_2}},
 \end{split}
   \label{eq:geometric_lift}
\end{equation}
which indeed is a geometric lift
of \eqref{eq:PL_transf}.

On the other hand,
the central fiber of the toric degeneration of a polygon space
may neither be Fano nor admit a small resolution,
and one can not apply the argument
of \cite{Nishinou-Nohara-Ueda_TDGCSPF}
directly to this case.
For example,
the moduli space of pentagons is isomorphic
to  $\bC \bP^1 \times \bC \bP^1$ for a suitable choice of $\br$,
and it degenerates to the Hirzebruch surface of degree 2,
which is not Fano.
In this case,
we need to consider contributions of sphere bubbles
to the potential function.
This case is studied in detail by Auroux
\cite{Auroux_MSTD, Auroux_SLFWCMS} 
and Fukaya, Oh, Ohta and Ono \cite{FOOO_toric_deg}.

This paper is organized as follows:
In Section \ref{sc:polygon_space},
we recall basic facts on polygon spaces.
We study the bending systems
in Section \ref{sc:bending_system}.
Theorem \ref{th:nohara_system} is proved
in Section \ref{sc:nohara_system}, and
Theorem \ref{th:piecewise-linear} is a corollary
to Proposition \ref{pr:piecewise-linear}.
Section \ref{sc:deg_in_stages}
is devoted to a construction of toric degenerations of $\Gr(2,n)$ in stages, 
which are used to construct a deformation of completely integrable systems.
Theorem \ref{th:toricdeg_nohara} is proved
in Section \ref{sc:deg_CIS}, and
Theorem \ref{th:X_Gamma} is proved
in Section \ref{sc:X_Gamma}.
In Section \ref{sc:potential_function},
we compute the potential functions
for the completely integrable 
systems on $\Gr(2,n)$
and prove Theorems
\ref{th:potential_function} and \ref{th:geometric_lift}.

{\bf Acknowledgment}:
We thank Takeo Nishinou
for collaboration
at an early stage of this research;
this paper is originally conceived
as a joint project with him.
We are also grateful to Hiroshi Konno for valuable discussion
which is reflected in the proof of Theorem \ref{th:toricdeg_nohara},
and to Megumi Harada and Laura Escobar for pointing an error 
in the proof of Lemma \ref{lem:ham-1st_stage}.
We thank the anonymous referee for carefully reading the manuscript
and suggesting several improvements.
This research is supported by
Grant-in-Aid for Young Scientists
(No.19740025, No.23740055, and No.20740037).


\section{Polygon spaces and Grassmannians}
 \label{sc:polygon_space}

In this section,
we fix notation and recall the relation
between polygon spaces and Grassmannians.

We identify the Lie algebra $\fraku(m)$
(resp. $\fraks \fraku(m)$)
of the Lie algebra of the unitary group $U(m)$
(resp. the special unitary group $SU(m)$)
with the space $\sqrt{-1} \mathfrak{u}(m)$ of Hermitian matrices 
(resp. the space $\sqrt{-1} \mathfrak{su}(m)$
of traceless Hermitian matrices).
The dual space $\fraku(m)^*$
(resp. $\fraks \fraku(m)^*$)
is identified with $\sqrt{-1} \fraku(m)$
(resp. $\sqrt{-1} \fraks \fraku(m)$)
by the invariant inner product
$\langle x, y \rangle = \tr(xy)$.
The moment map of the natural $U(m)$-action on $\mathbb{C}^m$
equipped with the standard symplectic structure is given by
\begin{equation}
  \mathbb{C}^m \longrightarrow \sqrt{-1} \mathfrak{u}(m),
  \quad
  z = \begin{pmatrix} z_1 \\ \vdots \\ z_m \end{pmatrix}
  \longmapsto
  \frac 12 zz^* = \frac 12 (z_i \overline{z}_j)_{i,j}
  \label{eq:unitary_mm}
\end{equation}
Recall from Introduction
that the polygon space $\pol$ is defined by
\[
  \pol
  = \Bigl. 
    \Bigl\{ \bld{x} = (x_1, \dots, x_n) \in \prod_{i=1}^n S^2(r_i) \, 
     \Bigm| \,
        x_1 + \dots + x_n = 0 \, \Bigr\}
    \Bigr/ SO(3),
\]
where $n \ge 3$ and
$
 r_i < r_1 + \dots + r_{i-1} + r_{i+1} + \dots + r_n
$
for
$
 i = 1, \ldots, n.
$
To describe a natural symplectic structure on $\pol$,
we identify $\Euc$ with $\sqrt{-1} \su$ by
\begin{equation} \label{eq:h}
\begin{array}{cccc}
 h : & \bR^3 & \simto & \sqrt{-1} \fraksu(2) \\
  & \vin & & \vin \\
 & \begin{pmatrix} x^1 \\ x^2 \\ x^3 \end{pmatrix}
 & \mapsto & 
  \begin{pmatrix}
    x^3 & x^1 - \sqrt{-1}x^2 \\
    x^1 + \sqrt{-1} x^2 & -x^3
  \end{pmatrix}.
\end{array}
\end{equation}
Then the $SO(3)$-action on $\Euc$ 
is induced from the (co)adjoint action of $SU(2)$, and 
the sphere $S^2(r_i)$ is identified with a (co)adjoint orbit
$\mathcal{O}_{r_i}$ of
$\mathrm{diag} (r_i,  -r_i)$,
which has the canonical
Kostant-Kirillov symplectic form
$\omega_{\mathcal{O}_{r_i}}$.
We equip $\mathcal{O}_{r_1} \times \dots \times 
\mathcal{O}_{r_n}$ with the symplectic form
$\sum_i \mathrm{pr}_i^* \omega_{\mathcal{O}_{r_i}}$, where 
$\mathrm{pr}_i : \mathcal{O}_{r_1} \times \dots \times 
\mathcal{O}_{r_n} \to \mathcal{O}_{r_i}$
is the $i$-th projection.
Then the diagonal $SU(2)$-action is Hamiltonian and its
moment map is given by
\[
  \mu : \mathcal{O}_{r_1} \times \dots \times \mathcal{O}_{r_n}
  \longrightarrow \sqrt{-1} \su^*, \quad
  (x_1, \dots, x_n) \longmapsto x_1 + \dots + x_n.
\]
Hence one has the following:
\begin{prop} \label{prop:symp_red'}
  $\pol$ is a symplectic reduction of
  $\mathcal{O}_{r_1} \times \dots \times 
\mathcal{O}_{r_n}$ by the diagonal $SU(2)$-action$:$
\[
  \pol = \Bigl. \Bigl. 
         \prod_{i=1}^n \mathcal{O}_{r_i} 
         \Bigr/\!\!\! \Bigr/_{\!\!\! 0} SU(2)
       = \mu^{-1}(0)/SU(2).
\]
\end{prop}
Let $\omega_{\pol}$ denote the induced symplectic form on $\pol$.
Identifying the symplectic reduction with a GIT quotient 
$(\prod_{i=1}^{n} \mathbb{CP}^1) \GIT  SL(2, \bC)$, 
we obtain a compatible complex structure on $(\pol, \omega_{\pol})$.

Next we recall a relation to the Grassmannian $\Gr(2,n)$.
Let $| \bld{r} | = \sum_i r_i$.
We consider the natural right $U(2)$-action
on the vector space $\bC^{n \times 2}$ of $n \times 2$ matrices.
From \eqref{eq:unitary_mm},
its moment map is given by
\begin{equation} \label{eq:U(2)-moment-map}
 \mu_{U(2)} : \mathbb{C}^{n \times 2} \to \sqrt{-1} \fraku (2), \quad
  \begin{pmatrix}
    z_1 & w_1 \\
    \vdots & \vdots \\
    z_n & w_n
  \end{pmatrix}
  \longmapsto \frac 12
  \sum_{i=1}^n \begin{pmatrix}
                  |z_i|^2 & \overline{z}_i w_i\\
                  z_i \overline{w}_i & |w_i|^2
               \end{pmatrix}.
\end{equation}
Then
\[
  \begin{pmatrix}
    z_1 & w_1 \\
    \vdots & \vdots \\
    z_n & w_n
  \end{pmatrix}
  \in \mu_{U(2)}^{-1} 
    \begin{pmatrix}
      |\bld r| & 0 \\
      0 & |\bld r| 
    \end{pmatrix}
\]
if and only if it satisfies
\begin{equation}
  \sum_i |z_i|^2 = \sum_i |w_i|^2 =  2 | \boldsymbol{r} |,
  \quad
  \sum_i z_i \overline{w}_i = 0.
\label{eq:grassmann}
\end{equation}
It follows that $\Gr(2,n)$ is a symplectic reduction of 
$\mathbb{C}^{n \times 2}$
by the $U(2)$-action:
\[
  \Gr(2,n) = \mathbb{C}^{n \times 2} \GIT_{|\bld{r}|} U(2)
  = \left. \mu_{U(2)}^{-1} 
    \begin{pmatrix}
      |\bld r| & 0 \\
      0 & |\bld r| 
    \end{pmatrix}
    \right/ U(2).
\]
We consider the moment map 
\begin{equation} \label{eq:SU(2)-moment-map}
  \mu_{SU(2)} : \bC^2 \longrightarrow \sqrt{-1} \su,
  \quad
  (z, w) \longmapsto \frac 14
     \begin{pmatrix}
       |z|^2 - |w|^2 & 2 \overline{z} w \\
       2 z \overline{w} & |w|^2 - |z|^2
     \end{pmatrix}
\end{equation}
of the standard $SU(2)$-action on $\mathbb{C}^2$.
The condition (\ref{eq:grassmann}) implies that 
\[
  \bigl( \mu_{SU(2)} (z_1, w_1), \dots, \mu_{SU(2)} (z_n, w_n) \bigr)
\]
gives a closed $n$-gon in $\sqrt{-1} \su \cong \mathbb{R}^3$ 
(i.e., it satisfies $\sum_i \mu_{SU(2)} (z_i, w_i) = 0$).
Since $\mu_{SU(2)} : \mathbb{C}^2 \to \sqrt{-1} \su$ is a quotient map by the 
diagonal $S^1$-action on $\mathbb{C}^2$, 
the quotient 
$T_{U(n)} \backslash \Gr(2,n)$ can be regarded as a moduli space
of polygons with fixed perimeter $| \bld{r} |$,
where $T_{U(n)} \subset U(n)$ is the maximal torus consisting of 
diagonal matrices.
Note that the $U(n)$-action on $\Gr(2,n)$ has the stabilizer 
of positive dimension.
The moment map $\mu_{T_{U(n)}} : \Gr(2,n) \to \mathbb{R}^n$ 
of the $T_{U(n)}$-action on $\Gr(2,n)$ is given by
\begin{equation} \label{eq:T_moment_map}
 \mu_{T_{U(n)}} : 
  \begin{bmatrix}
    z_1 & w_1 \\
    \vdots & \vdots \\
    z_n & w_n
  \end{bmatrix}
  \longmapsto
  \left( \frac{|z_1|^2 + |w_1|^2}2, \dots, 
    \frac{|z_n|^2 + |w_n|^2}2 \right).
\end{equation}
Since
$
 (|z_i|^2 + |w_i|^2)/2
  = 2 \| h^{-1} \circ \mu_{SU(2)} (z_i, w_i)\|_{\bR^3},
$
we have
\begin{prop}[Hausmann and Knutson
{\cite[(3.9)]{Hausmann-Knutson_PSG}}]
 \label{prop:symp_red}
The polygon space $\pol$ is isomorphic to
the symplectic reduction of $\Gr(2,n)$ by the $T_{U(n)}$-action$:$
\begin{equation}
  \pol \cong T_{U(n)} \GITl_{2 \bld r} \Gr(2,n)
  = T_{U(n)} \backslash \mu_T^{-1}(2 \boldsymbol{r}),
  \label{eq:symp_red}
\end{equation}
where the symplectic structure on $\Gr(2,n)$ is the Kostant-Kirillov
form on the (co)adjoint orbit of 
$\mathrm{diag}(|\bld{r}|, |\bld{r}|, 0, \dots, 0)$ in
$\sqrt{-1} \mathfrak{u}(n) (\cong \mathfrak{u}(n)^*)$.
\end{prop}

Propositions \ref{prop:symp_red'} and \ref{prop:symp_red} 
can be summarized as
$$
 \pol = \lb T_{U(n)} \GITl_{2 \br} 
         \bC^{n \times 2} \rb
         \GIT_{0} SU(2)
       = T_{U(n)} \GITl_{2 \br}
         \lb \bC^{n \times 2}
         \GIT_{|\br|} U(2) \rb
$$
where
$
 \prod_{i=1}^n S^2(r_i) = 
    T_{U(n)} \GITl_{2 \bld r} \mathbb{C}^{n \times 2}
$
and
$
  \Gr(2,n) =  \mathbb{C}^{n \times 2} \GIT_{|\bld{r}|} U(2).
$

\section{Bending Hamiltonians} \label{sc:bending_system}

Fix a convex $n$-gon $P \subset \mathbb{R}^2$
and call it the {\em reference $n$-gon}.
Let $e_1, \dots, e_n$ denote the sides of $P$ labeled in cyclic order. 
For an oriented diagonal $d$ of $P$, we write the corresponding 
diagonal of $\bld{x} \in \pol$ as $d(\bx)$.
If $d$ connect the $i$-th vertex and the $j$-th vertex for $i < j$, 
then $d(\bld{x})$ is given by
\[
  d(\bld{x}) = x_{i+1} + x_{i+2} + \dots + x_j
\]
or
\[
  d(\bld{x}) = x_{j+1} + \dots + x_{n} + x_1 + \dots + x_i
\]
depending on the orientation of $d$.
We define $\varphi_{d} : \pol \to \mathbb{R}$ 
to be the length function
\[
  \varphi_{d}(\bld{x}) = |d(\bld{x})|
\]
of the diagonal.
Kapovich and Millson \cite{Kapovich-Millson_SGPES} proved that
its Hamiltonian flow folds the polygon along the diagonal 
$d$ at a constant speed.
Thus $\varphi_{d}$ is called a {\it bending Hamiltonian}.

We say that two diagonals $d$ and $d'$ are {\it non-crossing}
if they do not intersect in the interior of $P$.

\begin{thm}[Kapovich and Millson \cite{Kapovich-Millson_SGPES},
Klyachko \cite{Klyachko_SP}]
If diagonals $d$ and $d'$ are non-crossing,
then $\varphi_{d}$ and $\varphi_{d'}$ are Poisson commutative.
Furthermore each choice of $(n-3)$ mutually non-crossing diagonals gives
a completely integrable system on $\pol$, and
the functions $\varphi_{d}$ give action variables.
\end{thm}

Note that such a choice of $(n-3)$ diagonals defines a triangulation of 
the reference $n$-gon $P$.
Here we consider only triangulations whose vertices coincides with those of $P$.
Let $\Gamma$ denote the dual graph of a given triangulation.
The graph $\Gamma$ is a trivalent
with $n$ leaves labeled by sides $e_1, \dots, e_n$ of $P$
and $(n-3)$ interior edges labeled by the diagonals $d_1, \dots, d_{n-3}$.
We often say a ``triangulation'' $\Gamma$
by abuse of notation.
We write the completely integrable system given by
$\Gamma$ as
\[
  \Phi_{\Gamma} = ( \varphi_{d_1}, \dots, \varphi_{d_{n-3}}) : \pol \longrightarrow 
  \mathbb{R}^{n-3},
\]
and call it the bending system associated to $\Gamma$.
The image 
$$
 \Delta_\Gamma(\br)
  := \Phi_\Gamma(\pol)
  \subset \bR^{n-3}
$$
is the convex polytope defined by triangle inequalities.

\begin{expl}
The triangulation given by
$d_{\alpha} = e_1 + \dots + e_{\alpha +1}$ ($\alpha = 1, \dots, n-3$)
is called the {\em caterpillar}
(see Figure \ref{fg:caterpillar}).
\begin{figure}
  \begin{center}
    \includegraphics*[height=2.5cm, bb= 0 0 101 82]{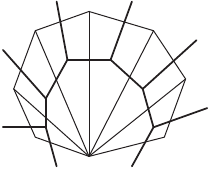}
  \end{center}
  \caption{The caterpillar}
  \label{fg:caterpillar}
\end{figure}
Let $(u_1, \dots, u_{n-3})$ be the coordinates on $\mathbb{R}^{n-3}$
corresponding to the bending Hamiltonians $\varphi_{d_{\alpha}}$.
The image 
$\Delta_{\Gamma}(\bld{r})$ is a polytope 
defined by triangle inequalities
\begin{equation}
  \begin{array}{rcccl}
    |r_1 - r_2 | &\le & u_1 &\le &r_1 + r_2,\\
    | u_1 - r_3 | &\le & u_2 &\le & u_1 + r_3,\\
    & & \vdots \\
    |u_{n-4} - r_{n-2} | &\le & u_{n-3} &\le & u_{n-4} + r_{n-2},\\
    |r_{n-1} - r_n | &\le &u_{n-3} &\le &r_{n-1} + r_n.
  \end{array}
  \label{eq:tri-ineq}
\end{equation}
\end{expl}

\begin{expl} \label{expl:pentagon}
Suppose $n=5$,
and consider a triangulation given by
$d_1 = e_1 + e_2$, $d_2 = e_1 + e_2 + e_3$.
If all side lengths $r_1, \dots, r_4$ are close,
then $\pol $ is isomorphic to $\mathbb{CP}^2$ blown up at four distinct points
(see \cite[Example 10.4]{Klyachko_SP}
or \cite[(6.3)]{Hausmann-Knutson_PSG}),
and the image $\Delta_{\Gamma}(\bld{r})$ is a heptagon 
shown in Figure \ref{fig:heptagon}.
\begin{figure}[h]
  \begin{center}
    \includegraphics*[height=5cm, bb= 0 0 189 161]{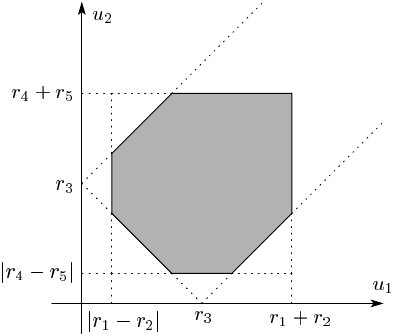}
  \end{center}
  \caption{The image of the bending system in Example \ref{expl:pentagon}.}
  \label{fig:heptagon}
\end{figure}
\end{expl}

\begin{expl} \label{ex:pentagon1}
Let $n=5$, and assume that
$r_1 > r_2 > 0$ and
\begin{equation*}
  r_1 + r_2 
  < \min ( -r_3+r_4+r_5, r_3-r_4+r_5, r_3+r_4-r_5).
\end{equation*}
In this case $\pol$ is isomorphic to 
$\mathbb{CP}^1 \times \mathbb{CP}^1$ 
(see \cite[(6.2)]{Hausmann-Knutson_PSG}
or \cite[Section 5]{Foth_MSP}).
We consider two triangulations $\Gamma_1$ and $\Gamma_2$ shown in 
Figure \ref{fig:triang_pent}.
The images $\Delta_{\Gamma_1} (\bold{r})$ and $\Delta_{\Gamma_2}(\bld{r})$ 
of the corresponding bending systems are shown 
in Figure \ref{fig:image1}.
$\Delta_{\Gamma_1}(\bld{r})$ is the moment polytope of the standard moment map 
on $\mathbb{CP}^1 \times \mathbb{CP}^1$, while
$\Delta_{\Gamma_2}(\bld{r})$ is that of the Hirzebruch surface 
$F_2 = \bP(\mathcal{O}_{\bP^1} \oplus \mathcal{O}_{\bP^1}(2))$ 
of degree two.
Note that $F_2$ is symplectomorphic to $\mathbb{CP}^1 \times \mathbb{CP}^1$
(but not isomorphic as complex manifolds).
\begin{figure}[h]
  \begin{center}
    \includegraphics*[bb = 0 0 212 79]{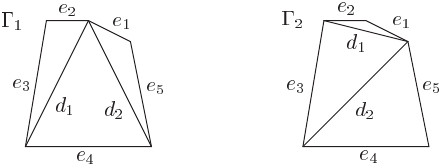}
  \end{center}
  \caption{Two triangulations in Example \ref{ex:pentagon1}}
  \label{fig:triang_pent}
\end{figure}
\begin{figure}[h]
  \begin{center}
    \includegraphics*[height=4.5cm, bb= 0 0 398 159]{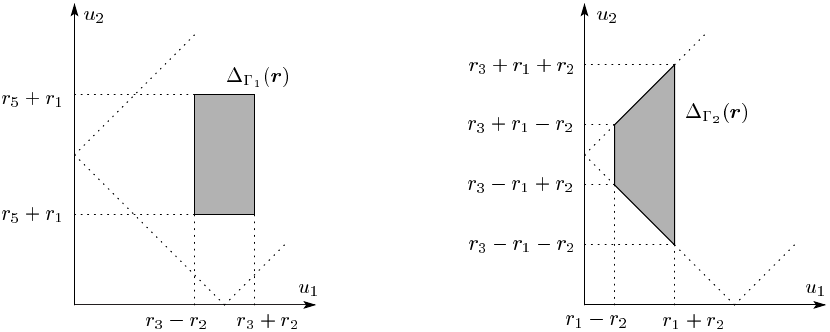}
  \end{center}
  \caption{The images $\Delta_{\Gamma_1}(\bld{r})$ and 
         $\Delta_{\Gamma_2}(\bld{r})$ in Example \ref{ex:pentagon1}}
  \label{fig:image1}
\end{figure}
\end{expl}

As we have seen in Example \ref{ex:pentagon1}, the bending system, 
and hence the corresponding image $\Delta_{\Gamma}(\bld{r})$, depends 
on the choice of a triangulation $\Gamma$. 
We compare the polytopes $\Delta_{\Gamma_i}(\bld{r})$
for different triangulations $\Gamma_1$ and $\Gamma_2$.
Recall that a {\it Whitehead move} (for triangulations) replaces a diagonal
of a quadrilateral with the other one (see Figure \ref{fig:whitehead1}).
\begin{figure}[h]
  \begin{center}
    \includegraphics*[bb = 0 0 188 72]{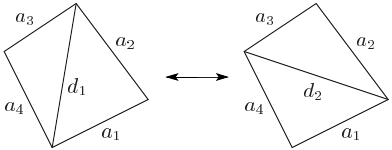}
  \end{center}
  \caption{A Whitehead move}
  \label{fig:whitehead1}
\end{figure}
Since triangulations can be transformed into each other by sequences
of Whitehead moves, it suffices to consider the case where 
$\Gamma_1$ and $\Gamma_2$ can be 
transformed to each other by a single Whitehead move.
Suppose that $\Gamma_1$ is transformed into $\Gamma_2$ by
a Whitehead move in a quadrilateral with sides $a_1, a_2, a_3, a_4$ 
replacing a diagonal $d= \pm(a_1+a_2)$ with $d'= \pm(a_2+a_3)$.
Note that $a_i$ is either a side $e_j$ of $P$ or a diagonal $d_{\alpha}$ 
contained in both of $\Gamma_1$ and $\Gamma_2$.
Let $u$, $u'$ and $u_i$ be the coordinates 
corresponding to $d$, $d'$ and $a_i$, respectively,
where we assume that $u_i=r_j$ is a constant if $a_i$ is a side $e_j$
of $P$.

\begin{prop}
  Under the above situation, the piecewise-linear transformation
  \begin{equation}
    \begin{split}
      u' & = u - \min(u_1+u_2, u_3+u_4) + \min(u_1+u_4, u_2+u_3) \\ 
         & = u + \min( u_1-u_2, u_2-u_1, u_3-u_4, u_4-u_3 )  \\
         & \qquad - \min( u_1-u_4, u_4-u_1, u_2-u_3, u_3-u_2 ) 
    \end{split}
    \label{eq:PL_transf1}
  \end{equation}
  gives a bijection 
  $\Delta_{\Gamma_1}(\bld{r}) \to \Delta_{\Gamma_2}(\bld{r})$.
  In particular, this map preserves the area of $\Delta_{\Gamma_i}(\bld{r})$ 
  and the number of integral points in $\Delta_{\Gamma_i}(\bld{r})$ 
  in the case where $\Delta_{\Gamma_i}(\bld{r})$ is an integral polytope
  $($i.e., $\bld{r} \in (\mathbb{Z}_{>0})^n)$.
\end{prop}

\begin{proof}
Since the triangle inequalities for $u$ are
\[
  \max (|u_1 - u_2|, |u_3 - u_4|)
  \le u \le
  \min (u_1 + u_2, u_3 + u_4),
\]
the length of the range of $u$ for fixed $u_1, \dots, u_4$ is
\[
  \min (u_1 + u_2, u_3 + u_4)  
   - \max (|u_1 - u_2|, |u_3 - u_4|).
\]
Since 
$\min \{a_i \, | \, i \in I \} + \min \{ b_j \,|\, j \in J \}
 = \min \{ a_i + b_j \,|\, i \in I, \, j \in J \}$, we have
\begin{align}
  &\min (u_1 + u_2, u_3 + u_4)  
   - \max (|u_1 - u_2|, |u_3 - u_4|) \notag \\
  & = \min (u_1 + u_2, u_3 + u_4) 
   + \min (u_1 - u_2, u_2 - u_1,  
           u_3 - u_4, u_4 - u_3) \notag \\
  &= \min \{ 2u_i, u_1 + u_2 + u_3 + u_4 - 2u_j \,| \,
    1 \le i, j \le 4 \}.
   \label{eq:length1}
\end{align}
Similarly the length of the range of $u'$ is
\begin{align*}
  &\min (u_1 + u_4, u_2 + u_3) 
   + \min (u_1 - u_4, u_4 - u_1,  
           u_2 - u_3, u_3 - u_2) \\
   &= \min \{ 2u_i, u_1 + u_2 + u_3 + u_4 - 2u_j \,| \,
    1 \le i, j \le 4 \},
\end{align*}
which is identical to \eqref{eq:length1}.
Hence \eqref{eq:PL_transf1} gives an area preserving transformation.
Moreover, if $u_1, \dots, u_4 \in \mathbb{Z}$,
then \eqref{eq:PL_transf1} is defined over $\mathbb{Z}$,
and hence the number of integral $u$ is also preserved.
\end{proof}

\begin{expl} \label{ex:pentagon2}
Let $\Gamma_1$ and $\Gamma_2$ be the triangulation 
in Example \ref{ex:pentagon1}.
$\Gamma_1$ can be transformed into $\Gamma_2$ by a two-step Whitehead move
shown in Figure \ref{fig:whitehead2}.
\begin{figure}[h]
  \begin{center}
    \includegraphics[bb= 0 0 296 79]{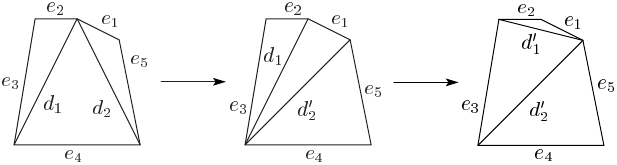}
    \end{center}
    \caption{Whitehead moves sending $\Gamma_1$ to $\Gamma_2$.}
  \label{fig:whitehead2}
\end{figure}
Let $u_{\alpha}$ and $u_{\alpha}'$ ($\alpha = 1,2$) 
be the coordinates corresponding 
to $d_{\alpha}$ and $d'_{\alpha}$, respectively.
The piecewise linear map corresponding to the first Whitehead move is 
given by
\[
  (u_1, u'_2) =
  (u_1, u_2 + \max (|u_1-r_1|, |r_4-r_5|)
            - \max (|u_1-r_4|, r_5-r_1)
  ).
\]
Since 
\begin{align*}
  \max (|u_1-r_1|, |r_4-r_5|) &= u_1 - r_1,\\
  \max (|u_1-r_4|, r_5-r_1) &= r_5- r_1
\end{align*}
on $r_3-r_2 < u_1 < r_3 + r_2$, 
the above map is given by
\[
  (u_1, u'_2) = (u_1, u_2+u_1-r_5)
\]
on $\Delta_{\Gamma_1}(\bld{r})$ (see Figure \ref{fig:PLtransf1}).
\begin{figure}[h]
  \begin{center}
    \includegraphics[bb= 0 0 311 121]{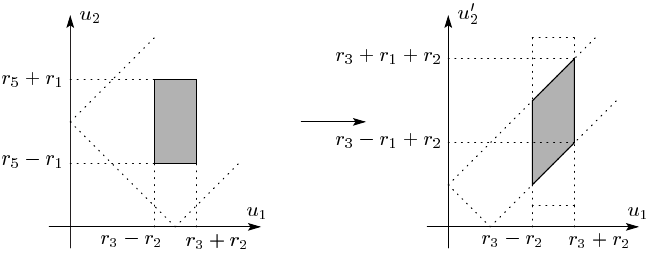}
  \end{center}
  \caption{The first piecewise-linear transformation in Example 
             \ref{ex:pentagon2}}
  \label{fig:PLtransf1}
\end{figure}

The piecewise-linear transformation for the second Whitehead move is
\[
  (u'_1, u'_2) = 
  (u_1 + \max (|u'_2-r_3|, (r_1-r_2) )
       - \max (|u'_2-r_1|, r_3-r_2 ),
  u'_2),
\]
which coincides with
\[
  (u'_1, u'_2) = 
  (u_1 - \min( u'_2- r_2, r_3-r_1),
   \, u'_2)
\]
on the image of $\Delta_{\Gamma_1}(\bld{r})$.
\begin{figure}[h]
  \begin{center}
    \includegraphics[bb = 0 0 333 120]{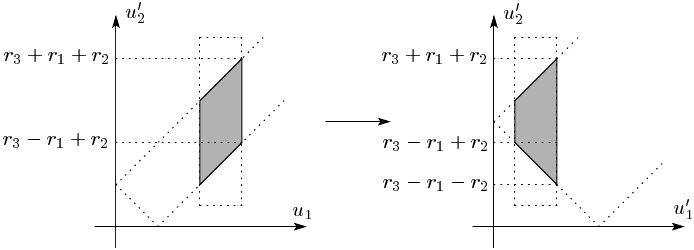}
    \end{center}
    \caption{The second piecewise-linear transformation in  
             Example \ref{ex:pentagon2}}
  \label{fig:PLtransf2}
\end{figure}
\end{expl}

\section{Completely integrable systems on $\Gr(2,n)$}
 \label{sc:nohara_system}

Hausmann and Knutson \cite{Hausmann-Knutson_PSG} proved that
the Gelfand-Cetlin system \cite{Guillemin-Sternberg_GCS} on $\Gr(2,n)$
induces the bending system on $\pol$ under the symplectic reduction
in the case where the triangulation is the caterpillar.
In this section we construct a completely integrable system on $\Gr(2,n)$
that induces the bending systems for each triangulation $\Gamma$.

Fix a triangulation $\Gamma$ of the reference $n$-gon $P$.
For each side $e_i$  ($i = 1, \dots, n$) and diagonal $d_{\alpha}$ 
($\alpha = 1, \dots, n-3$) we associate a subgroup of $U(n)$ as follows.
For a side $e_i$ we associate a subgroup isomorphic to $S^1$ given by
\[
  S^1_i = S^1_{e_i} = \begin{pmatrix}
           1_{i-1} && \\
           & S^1 & \\
           && 1_{n-i}
          \end{pmatrix}
  \subset U(n).
\]
For a diagonal  $d_{\alpha} = \sum_{i \in I_{\alpha}} e_i$, the corresponding 
subgroup $U_{\alpha} = U_{d_{\alpha}} \cong U(|I_{\alpha}|)$ is defined by
\[
  U_{d_{\alpha}} = \{ (g_{ij}) \in U(n) \, | \, 
           (g_{ij})_{i,j \in I_{\alpha}} \in U(|I_{\alpha}|) 
           \text{ and $g_{ij}=\delta_{ij}$ for $(i, j) \not\in I_{\alpha}^2$}  \}.
\]

\begin{rem} \label{rem:subgroups}
The non-crossing condition for $d_{\alpha}, d_{\beta}$ is equivalent to that 
$I_{\alpha}$ and $I_{\beta}$ satisfy either 
$I_{\alpha} \supset I_{\beta}$, $I_{\alpha} \subset I_{\beta}$, or
$I_{\alpha} \cap I_{\beta} = \emptyset$.
Hence 
each pair $G_1, G_2 \in \{U_1, \dots, U_{n-3}, S^1_1, \dots, S^1_n\}$
of subgroups satisfies either $G_1 \subset G_2$, $G_1 \supset G_2$, or 
$G_1 \cap G_2 = \{ 1 \}$ in $U(n)$.
\end{rem}

Recall that the moment map
$\mu_{U(n)} : \Gr(2,n) \to \mathfrak{u}(n)^* \cong 
\sqrt{-1} \mathfrak{u}(n)$ 
of the $U(n)$-action is given by
\[
  \begin{bmatrix}
    z_1 & w_1 \\
    \vdots & \vdots \\
    z_n & w_n
  \end{bmatrix}
  \longmapsto \frac 12
  \begin{pmatrix}
    z_1 & w_1 \\
    \vdots & \vdots \\
    z_n & w_n
  \end{pmatrix}
  \begin{pmatrix}
    \overline{z}_1 & \dots & \overline{z}_n \\
    \overline{w}_1 & \dots & \overline{w}_n
  \end{pmatrix}
  =
  \left( \frac{z_i \overline{z}_j + w_i \overline{w}_j}2 
  \right)_{i,j}.
\]
Then the moment maps for $S^1_{e_i}$ and $U_{d_{\alpha}}$
are given by
\[
  \psi_{e_i} : \Gr(2,n) \longrightarrow \mathbb{R},
  \quad \psi_{e_i} (z,w) = \frac{|z_i|^2+|w_i|^2}2
\]
and 
\[
  \mu_{U_{\alpha}}: \Gr(2,n) \to \sqrt{-1} \mathfrak{u}(|I_{\alpha}|),
  \quad 
  \mu_{U_{\alpha}}(z, w) = 
  \left( 
  \frac{z_i \overline{z}_j + w_i \overline{w}_j}2 
  \right)_{i, j \in I_{\alpha}},
\] 
respectively.
Note that $\mu_{U(n)}$ embeds $\Gr(2,n)$ into 
$\sqrt{-1} \mathfrak{u}(n)$ as the adjoint orbit of
$\mathrm{diag}(|\bld{r}|, |\bld{r}|, 0, \dots, 0)$.
In particular, $\mu_{U(n)}(z,w)$ has rank two for any $[z,w] \in \Gr(2,n)$, 
and hence the rank of $\mu_{U_{\alpha}}(z,w)$ is  at most two.
Let
$\lambda_{\alpha,1} (z, w) \ge \lambda_{\alpha,2} (z, w) \ge 0$
be the first and second eigenvalues of $\mu_{U_{\alpha}}(z, w)$.


Recall that the moment map is a Poisson morphism:

\begin{prop}[{\cite{Kostant_OSSRT}}] \label{pr:moment_Poisson}
Let $G$ be a compact Lie group acting on a symplectic manifold
$(X, \omega)$ with a moment map $\mu : X \to \mathfrak{g}^*$.
For $f_1, f_2 \in C^{\infty}(\mathfrak{g}^*)$, we have
\[
  \{ \mu^* f_1, \mu^* f_2 \}_{X}
  = \mu^* \{ f_1, f_2 \}_{\mathfrak{g}^*},
\]
where $\{ \,\, , \, \}_{\mathfrak{g}^*}$ is
the natural Poisson structure on $\mathfrak{g}^*$.
\end{prop}


This immediately yields the following:

\begin{corollary} \label{cr:involutive1}
Let $G$ be a compact Lie group acting on a symplectic manifold
$(X, \omega)$ with a moment map $\mu : X \to \mathfrak{g}^*$,
and $f_1, f_2 \in C^{\infty}(\mathfrak{g}^*)$
be smooth functions on $\frakg^*$.
If either $f_1$ or $f_2$ is
$\mathrm{Ad}^*(G)$-invariant
{\rm (}i.e., if it is constant on each 
coadjoint orbit{\rm )},
then $\mu^* f_1$ and $\mu^* f_2$ are 
Poisson commutative$:$
\[
  \{ \mu^* f_1, \mu^* f_2 \}_{X} = 0.
\]
\end{corollary}

We also have the following:

\begin{lem} \label{lm:involutivity1}
Let $(X,\omega)$ be a symplectic manifold, and assume that two Lie groups
$G_1$, $G_2$ act on $(X, \omega)$ in Hamiltonian fashion 
with moment maps
$\mu_i : X \to \mathfrak{g}_i^*$, $i=1, 2$.
If $G_1$ and $G_2$ satisfy either $G_1 \subset G_2$ or  
$G_1 \cap G_2 = \{ 1 \}$ in the group of symplectomorphisms,
then 
\[
  \{ \mu_1^* f_1, \mu_2^* f_2 \}_X = 0
\]
for $\mathrm{Ad}^*(G_i)$-invariant functions 
$f_i \in C^{\infty}(\mathfrak{g}_i^*)^{G_i}$,
$i=1,2$.
\end{lem}

\begin{proof}
We first assume that $G_1 \subset G_2$.
Then $\mu_1$ is a composition of 
$\mu_2 : X \to \mathfrak{g}_2^*$ and the natural projection
$p: \mathfrak{g}_2^* \to \mathfrak{g}_1^*$.
From Proposition \ref{pr:moment_Poisson} we have
\[
  \{ \mu_1^* f_1, \mu_2^* f_2 \}_X = 
  \{ \mu_2^* p^* f_1, \mu_2^* f_2 \}_X =
  \mu_2^* \{p^* f_1, f_2 \}_{\mathfrak{g}_2^*} = 0.
\]
Next we consider the second case.
Then the moment map of the action of $G = G_1 \times G_2$ 
is given by
\[
  \mu = (\mu_1, \mu_2) : X \longrightarrow  
  \mathfrak{g}_1^* \oplus \mathfrak{g}_2^*,
\]
and we have $\mu_i = p_i \circ \mu$, where 
$p_i : \mathfrak{g}_1^* \oplus \mathfrak{g}_2^* \to \mathfrak{g}_i^*$
is the $i$-th projection.
Since $\mu_i^* f_i = \mu^* (p_i^* f_i)$ and
$p_i^* f_i \in C^{\infty} (\mathfrak{g}_1^* \oplus \mathfrak{g}_2^*)^G$, 
the Poisson commutativity follows from the argument in the first case.
\end{proof}

This shows the following:

\begin{prop} \label{prop:Poisson_commutative}
The functions $\psi_{e_i}$, $i=1, \dots , n$
and $\lambda_{\alpha,j}$, $\alpha=1, \dots, n-3$, $j=1,2$
are mutually Poisson commutative.
\end{prop}

\begin{proof}
This follows from
Remark \ref{rem:subgroups},
Proposition \ref{pr:moment_Poisson},
Lemma \ref{lm:involutivity1}, and
the fact
that $\lambda_{\alpha,j}$ and $\psi_{e_i}$ are
pull-backs of  invariant functions.
\end{proof}

Since the number 
$\# \{\lambda_{\alpha,j}, \psi_{e_i} \} = 3n-6$
of functions we have obtained is
larger than $\frac 12 \dim_{\mathbb R} \Gr(2,n) = \N$,
these functions can not be functionally independent.
In fact there is one linear relation 
for each triangle in the triangulation.
For example, for a triangle whose edges are three diagonals, say, 
$d_1$, $d_2$, and $d_3 = d_1 + d_2$, $U_1 \times U_2$ is 
a subgroup of $U_3$, and the moment map $\mu_{U_3}$ has the form
\[
  \mu_{U_3} = \begin{pmatrix}
                 \mu_{U_1} & * \\
                 * & \mu_{U_2}
              \end{pmatrix}.
\]
Hence we have
\begin{align*}
  \lambda_{3,1} + \lambda_{3, 2} 
  &= \tr \mu_{U_3}\\
  &= \tr \mu_{U_1} + \tr \mu_{U_2}\\
  &= \lambda_{1,1} + \lambda_{1, 2}
   + \lambda_{2,1} + \lambda_{2, 2}.
\end{align*}
Set $\psi_{d_{\alpha}} = \lambda_{\alpha, 2}$ for each diagonal $d_{\alpha}$. 
We claim that 
\[
  \Psi_{\Gamma} := (\psi_{d_1}, \psi_{d_2}, \dots, \psi_{d_{n-3}},
                   \psi_{e_1}, \dots, \psi_{e_{n-1}})
  : \Gr(2,n) \longrightarrow \bR^\N
\]
is a completely integrable system on $\Gr(2,n)$.

\begin{prop} \label{prop:reduction_CIS}
  $(\psi_{d_1}, \psi_{d_2}, \dots, \psi_{d_{n-3}})$ induces 
  the bending system $\Phi_{\Gamma}$ on $\pol$  by the symplectic reduction
  $(\ref{eq:symp_red})$ up to sign and additive constants.
\end{prop}

\begin{proof}
Note that
the set of non-zero eigenvalues of $AB$
for not necessarily square matrices $A$ and $B$
is bijective with the set of non-zero eigenvalues of $BA$.
It follows that the first and second eigenvalues 
$\lambda_{\alpha,1}, \lambda_{\alpha,2}$ of
\[
  \mu_{U_{\alpha}} (z, w) = \frac 12
    (z_i, w_i)_{i \in I_{\alpha}} 
    (z_i, w_i)_{i \in I_{\alpha}}^*
  \in \sqrt{-1} \mathfrak{u} (|I_{\alpha}|)
\]
coincide with those of
\begin{equation}
  \frac 12
    (z_i, w_i)_{i \in I_{\alpha}}^* 
    (z_i, w_i)_{i \in I_{\alpha}}
  = \frac 12
  \sum_{i \in I_{\alpha}}
  \begin{pmatrix}
    |z_i|^2 & \overline{z}_i w_i \\
    z_i \overline{w}_i & |w_i|^2
  \end{pmatrix}
  \in \sqrt{-1} \mathfrak{u}(2).
  \label{eq:shifted_diagonal}
\end{equation}
Since (\ref{eq:shifted_diagonal}) is $T_{U(n)}$-invariant, 
its eigenvalues $\lambda_{\alpha,j}$ descends to functions on $\pol$.
Recall from Section \ref{sc:polygon_space}
that sides of the polygon are given by
\[
  \mu_{SU(2)}(z_i, w_i) = \frac 12 \begin{pmatrix}
    (|z_i|^2-|w_i|^2)/2 & \overline{z}_i w_i \\
    z_i \overline{w}_i & (|w_i|^2-|z_i|^2)/2
  \end{pmatrix},
\]
considered as an element of $\bR^3$
by the isomorphism $h : \bR^3 \simto \su$.
Then (\ref{eq:shifted_diagonal}) can be written as
\[
  \sum_{i\in I_{\alpha}} \mu_{SU(2)} (z_i, w_i)
  + \frac 14
  \sum_{i\in I_{\alpha}} \begin{pmatrix}
    |z_i|^2+|w_i|^2 & 0 \\
    0 & |z_i|^2+|w_i|^2
  \end{pmatrix},
\]
whose second term is a constant
$\sum_{i \in I_{\alpha}} \mathrm{diag} (r_i, r_i)$ 
on the level set $\mu_{T_{U(n)}}^{-1}(2 \bld{r})$,
while the first term 
$\sum_{i \in I_{\alpha}} \mu_{SU(2)} (z_i, w_i)$, which is the $\alpha$-th
diagonal, has eigenvalues $\pm \varphi_{\alpha}$.
Hence we have
\[
  \lambda_{\alpha,2} = - \varphi_{\alpha} 
  + \sum_{i \in I_{\alpha}} r_i
\]
on $\pol$.
\end{proof}

The next lemma completes the proof of Theorem \ref{th:nohara_system}.
\begin{lem} \label{lem:fun-indep}
  The functions in $\Psi_{\Gamma}$ are functionally independent.
\end{lem}

\begin{proof}
For a function $f$, let $\xi_f$ denote its Hamiltonian vector field.
From Proposition \ref{prop:reduction_CIS} and the fact that 
the bending Hamiltonians are functionally independent, 
$\xi_{\psi_{d_1}}, \dots, \xi_{\psi_{d_{n-3}}}$
are linearly independent and transverse to $T_{U(n)}$-orbits
on an open dense subset of a level set $\mu_{T_{U(n)}}^{-1}(\bld{r}) \subset \Gr(2,n)$ 
for generic $\bld{r}$. 
On the other hand, 
$\xi_{\psi_{e_1}}, \dots, \xi_{\psi_{e_{n-1}}}$ give basis of 
tangent spaces of $T_{U(n)}$-orbits, and thus 
$\{ \xi_{\psi_{d_{\alpha}}}, \xi_{\psi_{e_i}} \}_{\alpha, i}$ 
are linearly independent on an open dense subset.
\end{proof}

\begin{rem}
  From the facts that the bending Hamiltonians $\varphi_{\alpha}$ are 
  action variables and  $\psi_{e_i}$ are moment maps of 
  $S^1$-actions,
  the functions $(\psi_{d_{\alpha}}, \psi_{e_i})_{\alpha, i}$ are also
  action variables.
\end{rem}

Recall that lengths of sides and diagonals
of the polygon are given by 
$
 r_i = \psi_{e_i}/2
$
and
$
 \varphi_{d_{\alpha}}
 = \sum_{i \in I_{\alpha}} \psi_{e_i}/2 - \psi_{d_{\alpha}}.
$
Let $(u_{e_1}, \dots, u_{e_{n-1}}, u_{d_1}, \dots, u_{d_{n-3}})$ be the coordinates
on $\mathbb{R}^{\N}$ corresponding to $\psi_{e_i}$ and $\psi_{d_{\alpha}}$, and 
define another coordinates corresponding to length functions for 
$a \in \{e_1, \dots, e_n, d_1, \dots, d_{n-3} \}$ by
\begin{equation}
  u(a) = \begin{cases}
    \frac 12 u_{e_i}, & a = e_i \quad (i=1, \dots, n-1),\\
    |\bld{r}| - \frac 12 \sum_{i=1}^{n-1} u_{e_i}, & a = e_n,\\
    -u_{d_{\alpha}} + \frac 12 \sum_{i \in I_{\alpha}} u_{e_i},
    & a = d_{\alpha}.
  \end{cases}
  \label{eq:length_fn}
\end{equation}

\begin{theorem}
Let  $\Delta_{\Gamma}$ be the polytope in $\bR^\N$
defined by triangle inequalities
\begin{equation} \label{eq:triangle}
  | u(a) - u(b)| \le u(c) \le u(a) + u(b)
\end{equation}
for each triangle in $\Gamma$ with sides $a, b, c$.
Then $\Delta_\Gamma$ is the moment polytope
of the integrable system$;$
$
 \Delta_\Gamma = \Psi_{\Gamma}(\Gr(2,n)).
$
\end{theorem}

\begin{proof}
Note that there exists a polygon with 
prescribed side and diagonal lengths
if they satisfy triangle inequalities.
Since the quotient $\Gr(2,n)/T_{U(n)}$ can be regarded 
as a moduli space of polygons with fixed perimeter, 
the induced map
$\Psi_{\Gamma} : \Gr(2,n)/T_{U(n)} \to \Delta_{\Gamma}$ 
is surjective, which proves the theorem.
\end{proof}

The image $\Delta_{\Gamma}(\bld{r})$ of the polygon
space $\pol$ is a subset of $\Delta_{\Gamma}$ defined by
$(u(e_1), \dots, u(e_n)) = \br$.
Furthermore we have the following:

\begin{prop} \label{pr:piecewise-linear}
Let $\Gamma_1$ and  $\Gamma_2$ be two triangulations which are 
transformed into each other by a Whitehead move replacing $d$ with $d'$.
Then $\Delta_{\Gamma_1}$ is transformed into $\Delta_{\Gamma_2}$ 
by a piecewise-linear transformations 
$(\ref{eq:PL_transf})$ with respect to $(u(e_i), u(d_{\alpha}))_{i, \alpha}$.
This map is defined over $\mathbb{Z}$ with respect to the coordinates
$(u_{e_i}, u_{d_{\alpha}})_{i, \alpha}$ if $|\bld{r}| \in \mathbb{Z}$.
Hence the volume and the number of integral points in the case where
$\Delta_{\Gamma}$ is integral are independent of the choice of $\Gamma$.
\end{prop}

\begin{proof}
It is obvious that the piecewise linear map given by (\ref{eq:PL_transf})
sends $\Delta_{\Gamma_1}$ to $\Delta_{\Gamma_2}$ preserving the volumes.
We need to show that this map is defined over $\mathbb{Z}$,
since the transformation (\ref{eq:length_fn}) is not defined 
over $\mathbb{Z}$.

Let $a_1, \dots, a_4 \in \{e_1, \dots, e_n, d_1, \dots, d_{n-3} \}$
be the sides of the quadrilateral having $d$ and $d'$ as its diagonals,
and take $I_{a_1}, \dots, I_{a_4} \subset \{1, \dots, n\}$ such that 
$\sum_{i \in I_{a_k}} e_i = \pm a_k$ and 
$I_{a_1} \sqcup \dots \sqcup I_{a_4} = \{1, \dots, n \}$.
We set $I_d = I_{a_1} \cup I_{a_2}$ and 
$I_{d'} = I_{a_1} \cup I_{a_4}$ so that
\[
  \sum_{i \in I_d} e_i = \pm d, \quad
  \sum_{i \in I_{d'}} e_i = \pm d'.
\]
Define $v(a) = \sum_{j \in I_a} u_{e_j}/2 = \sum_{i \in I_a} u(e_i)$ 
for each $a = a_1, \dots, a_4, d, d'$. 
Then $u(a) \pm v(a)$ have integral coefficients 
with respect to $u_{e_i}$ and $u_{d_{\alpha}}$.
We also note that 
\[
  v(a_1) + \dots + v(a_4) = \sum_{i=1}^n u(e_i) = |\bld{r}|.
\]
Then we have
\begin{multline*}
  \min (u(a_1)+u(a_2), u(a_3)+u(a_4))\\
  = \min \bigl( u(a_1)+v(a_1)+u(a_2)+v(a_2), 
                 u(a_3)-v(a_3)+u(a_4)-v(a_4) + |\bld{r}| \bigl)\\
  - v(a_1) - v(a_2), 
\end{multline*}
and the first term in the right hand side is defined over $\mathbb{Z}$.
Similarly we have
\begin{multline*}
  \min (u(a_1)+u(a_4), u(a_2)+u(a_3))\\
  = \min \bigl( u(a_1)+v(a_1)+u(a_4)+v(a_4), 
                 u(a_2)-v(a_2)+u(a_3)-v(a_3) + |\bld{r}| \bigl)\\
  - v(a_1) - v(a_4).
\end{multline*}
Thus (\ref{eq:PL_transf}) can be written as
\[
  u(d') + v(a_1) + v(a_4) =
  u(d) + v(a_1) + v(a_2) + g(u)
\]
for some $g(u) \in \mathbb{Z}[u_{e_i}, u_{d_{\alpha}}]$.
Since $v(a_1) + v(a_2) = v(d)$ and  $v(a_1) + v(a_4) = v(d')$,
the above coordinate change is defied over $\mathbb{Z}$.
\end{proof}

\begin{expl}[cf.~Hausmann and Knutson \cite{Hausmann-Knutson_PSG}]
Suppose that $\Gamma$ is the caterpillar given by
$d_{\alpha} = e_1 + e_2 + \dots + e_{\alpha +1}$, $\alpha = 1, \dots, n-3$.
Then the subgroups $U_{d_{\alpha}} \cong U(\alpha+1)$ satisfy
\[
  U_{d_{n-3}} \supset U_{d_{n-4}} \supset \dots \supset U_{d_1}
  \supset S^1_{e_1}.
\]
The first and the second eigenvalues
of the moment map $\mu_{U_{\alpha}}$ is given by
\begin{equation}
  \lambda^{(\alpha+1)}_1 := \lambda_{\alpha, 1} 
  = - \psi_{d_{\alpha}} + \sum_{i=1}^{\alpha+1} \psi_{e_i},
  \quad
  \lambda^{(\alpha+1)}_2 := \lambda_{\alpha, 2} = \psi_{d_{\alpha}}.
  \label{eq:coord_GC}
\end{equation}
We also consider the action of $U_{-e_n} \cong U(n-1)$ corresponding to
$I_{-e_n} = \{ 1, 2, \dots, n-1 \}$.
The eigenvalues of the corresponding moment map $\mu_{U_{-e_n}}$
gives functions of eigenvalues 
$\lambda^{(n-1)}_1 \ge \lambda^{(n-1)}_2 \ge 0$.
Since $\lambda^{(n-1)}_1 = |\bld{r}|$ is constant and
\[
  \lambda^{(n-1)}_2 = \sum_{i=1}^{n-1} \psi_{e_i} - |\bld{r}|,
\] 
$\Psi_{\Gamma}$ is equivalent to the Gelfand-Cetlin system
$(\lambda_j^{(k)})_{j,k} : \Gr(2,n) \to \bR^\N$,
where we set $\lambda^{(1)}_1 = \psi_{e_1}$.
It is easy to check that the triangle inequalities are equivalent to
the inequalities for Gelfand-Cetlin patterns
\begin{equation}
\begin{alignedat}{17}
   {\hbox to1.65em{$| \bld{r}|$}} &&&& \llmd {n-1}2 \\
  & \uge && \dge && \uge \\
  && \llmd {n-2}1 &&&& \llmd{n-2}2 \\
  &&& \uge && \dge && \uge \\
  &&&& \dndots &&&& \dndots  \\
  &&&&& \uge && \dge && \uge \\
  &&&&&& \llmd 31 &&&& \llmd 32 &&&& {\hbox to1.65em{$\,\, 0$}} \\
  &&&&&&& \uge && \dge && \uge && \dge \\
  &&&&&&&& \llmd 21 &&&& \llmd 22 \\
  &&&&&&&&& \uge && \dge &&&&&&& \\
  &&&&&&&&&& \llmd 11 &&&&&&&&& 
\end{alignedat}
\label{GC-pattern}
\end{equation}
\end{expl}

\section{Degenerations of Grassmannians in stages}
 \label{sc:deg_in_stages}

Recall that the Pl\"ucker embedding 
$\Gr(2,n) \hookrightarrow \mathbb{P}(\bigwedge^2 \mathbb{C}^n)$
is given by
\[
  \begin{bmatrix}
    z_1 & w_1\\
    \vdots & \vdots \\
    z_n & w_n
  \end{bmatrix}
  \longmapsto 
  [ Z_{ij}; 1 \le i<j \le n],
  \quad
  Z_{ij} = \det \begin{pmatrix}
                 z_i & w_i \\
                 z_j & w_j
                \end{pmatrix},
\]
and the image is defined by the Pl\"ucker relations
\[
  p_{ijkl} (Z) = Z_{ij}Z_{kl} - Z_{ik}Z_{jl} + Z_{il}Z_{jk} = 0,
  \quad
  i<j<k<l.
\]
Toric degenerations of $\Gr(2,n)$ are given by deforming 
the Pl\"ucker relations into binomials.
Speyer and Sturmfels \cite{Speyer-Sturmfels_TG}
proved that each toric degeneration of $\Gr(2,n)$
corresponds to a triangulation of the reference polygon $P$.
In this section we construct a multi-parameter deformation of $\Gr(2,n)$ 
which is an extension of the one-parameter family
in \cite{Speyer-Sturmfels_TG}.

Fix a triangulation of the reference $n$-gon $P$, and let $\Gamma$ be
its dual graph.
We choose a numbering and orientations of the diagonals 
$d_{\alpha} = \sum_{i \in I_{\alpha}} e_i$ 
in such a way that either $I_{\alpha} \supset I_{\beta}$
or $I_{\alpha} \cap I_{\beta} = \emptyset$ is satisfied if $\alpha < \beta$.
In particular,
we assume that $|I_1|=n-2$ and
$I_{\alpha} \subset I_1$ for all $\alpha \ge 2$. 
For two leaves $i$, $j$ of $\Gamma$, let $\gamma (i,j)$ be the unique path in $\Gamma$
connecting $i$ and $j$ (see Figure \ref{fig:graphpath}).
\begin{figure}[h]
  \begin{center}
    \includegraphics*[bb= 0 0 98 86]{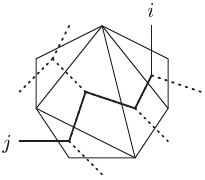}
  \end{center}
  \caption{A path $\gamma(i,j)$ connecting the $i$-th and $j$-th leaves}
  \label{fig:graphpath}
\end{figure}
We introduce a parameter $t_{\alpha}$ for each diagonal $d_{\alpha}$.
We define a weight 
$\bld{w}_{ij}^{\Gamma} = (w_{ij,1}^{\Gamma},\dots, w_{ij,n-3}^{\Gamma})$
of $Z_{ij}$ by
\[
  w_{ij,\alpha}^{\Gamma} = 
    \begin{cases}
      \frac 12 & \text{if $\gamma(i,j)$ intersects $d_{\alpha}$},\\
      0 & \text{otherwise},
    \end{cases}
\]
and set
\[
  \bld{t}^{\bld{w}_{ij}^{\Gamma}} 
  = \prod_{\alpha=1}^{n-3} t_{\alpha}^{w_{ij,\alpha}^{\Gamma}}
  = \prod_{\alpha} \sqrt{t_{\alpha}},
\]
where the product on the right hand side is taken over all 
$\alpha \in \{1, \dots, n-3 \}$
such that $d_{\alpha}$ crosses $\gamma (i,j)$.
For a polynomial $p(Z_{ij})$, we define
$w^{\Gamma}_{\alpha}(p)$ to be the maximum of weights
of monomials in $p$ with respect to $t_{\alpha}$, and set
\[
  \widetilde{p}(Z_{ij}, t_{\alpha}) =
  \prod_{\alpha} t_{\alpha}^{w^{\Gamma}_{\alpha}(p)} 
  \cdot p(\bld{t}^{-\bld{w}_{ij}^{\Gamma}}Z_{ij}).
\]
Then $\widetilde{p}$ is a polynomial in $Z_{ij}$ and $t_{\alpha}$ for each 
Pl\"ucker relation $p=p_{ijkl}$.
The degeneration $f_{\Gamma} : \mathfrak{X}_{\Gamma} \to \bC^{n-3}$
of $\Gr(2,n)$ corresponding to $\Gamma$ is given by
\[
   \mathfrak{X}_{\Gamma} = 
  \Bigl\{ (Z, \bld{t}) \in \mathbb{P}(
  \textstyle{\bigwedge^2} \mathbb{C}^n) \times 
  \mathbb{C}^{n-3} \, \Bigm| \,  \widetilde{p}_{ijkl} (Z, \bld{t}) = 0 ,
  \, i<j<k<l
  \Bigr\} .
\]
The restriction $\widetilde{p}_{ijkl}(Z, 0)$ to $t_1 = \dots = t_{n-3} = 0$ is
a binomial, and hence
the central fiber 
$X_{\Gamma, 0} = f_{\Gamma}^{-1}(0, \dots, 0)$ is a toric variety
(see \cite[Section 4]{Speyer-Sturmfels_TG}
or \cite[Section 4]{Howard-Manon-Millson}).
We will see in the next section that
$\Delta_{\Gamma} = \Psi_{\Gamma}(\Gr(2,n))$
is the moment polytope of $X_{\Gamma, 0}$.

\begin{rem}
The restriction of $\mathfrak{X}_{\Gamma}$ to the diagonal
$t_1 = \dots = t_{n-3} = t$ is the family constructed
in \cite{Speyer-Sturmfels_TG,Howard-Manon-Millson}.
\end{rem}

\begin{expl}
Let $n=5$ and consider the triangulation given by
$d_1 = e_1+e_2$ and $d_2 = e_1+e_2+e_3$.
Then the defining equations of $\mathfrak{X}_{\Gamma}$ are given by
\[
  \left\{
  \begin{aligned}
    t_1 Z_{12}Z_{34} - Z_{13}Z_{24} + Z_{14}Z_{23} &= 0,\\
    t_1 Z_{12}Z_{35} - Z_{13}Z_{25} + Z_{15}Z_{23} &= 0,\\
    t_1 t_2 Z_{12}Z_{45} - Z_{14}Z_{25} + Z_{15}Z_{24} &= 0,\\
    t_2 Z_{13}Z_{45} - Z_{14}Z_{35} + Z_{15}Z_{34} &= 0,\\
    t_2 Z_{23}Z_{45} - Z_{24}Z_{35} + Z_{25}Z_{34} &= 0.
  \end{aligned}
  \right.
\]
\end{expl}

To see the degeneration in more detail, 
we introduce the following notation.
For an $m$-gon $P'$ with sides labeled by $1, \dots, m$, 
we define 
\[
  \widetilde{\Gr}_{P'}
  = \mathbb{C}^{m \times 2} \GIT_0 SU(2)
\]
to be a cone over the Grassmannian
$$
 \Gr_{P'}
  = \bC^{m \times 2} \GIT_{|\br|} U(2),
$$
so that
$
 \Gr_{P'}
  = \widetilde{\Gr}_{P'} \GIT_{|\br|} S^1 \cong \Gr(2,m),
$
and write elements of 
$\Gr_{P'}$ or $\widetilde{\Gr}_{P'}$ as
\[
  \begin{bmatrix} 
     z_1^{P'} & w_1^{P'} \\ \vdots & \vdots \\ z_m^{P'} & w_m^{P'} 
  \end{bmatrix}.
\]
Let
$\widetilde{\Gr}_{P'} \hookrightarrow V_{P'} := \bigwedge^2 \mathbb{C}^m$
and $\Gr_{P'} \hookrightarrow \bP (V_{P'})$
be the Pl\"ucker embeddings given by
\[
  Z_{ij}^{P'} = \det 
    \begin{pmatrix}
      z_i^{P'} & w_i^{P'} \\ z_j^{P'} & w_j^{P'} 
    \end{pmatrix}.
\]
Note that if $P'$ is a triangle, then 
$\widetilde{\Gr}_{P'} = V_{P'} = \bigwedge^2 \bC^3 \cong \bC^3$ and 
$\Gr_{P'} \cong \bP^2$.
For a triangulation $\Gamma'$ of $P'$, let 
$\mathfrak{X}_{\Gamma'} \to \bC^{m-3}$ 
and $\widetilde{\mathfrak{X}}_{\Gamma'} \to \bC^{m-3}$ denote the 
corresponding toric degenerations of $\Gr_{P'}$ and 
$\widetilde{\Gr}_{P'}$, respectively.

We fix a diagonal $d = d_{\alpha}$ in $\Gamma$ and consider a one-parameter subfamily 
$f'_{\Gamma} : \frakX'_{\Gamma} \to \mathbb{C}$
of $f_{\Gamma} : \frakX_{\Gamma} \to \mathbb{C}^{n-3}$ defined by
$t_{\beta}=1$ for all $\beta \ne \alpha$.
Suppose that the diagonal $d$ connects the $p$-th vertex
and the $q$-th vertex ($p<q$), 
and set
\begin{align*}
  I_+ &= \{ p+1, p+2, \dots, q \}, \\
  I_- &= \{ 1, \dots, p, q+1, \dots, n\} = \{1, \dots, n \} \setminus I_+.
\end{align*}
Then $d_+ = \sum_{i \in I_+} e_i$ or  $d_- = \sum_{i \in I_-} e_i = -d_+$
coincides with $d$, and $U_{d_+} \cap U_{d_-} = \{1\}$ in $U(n)$.
The defining equations 
of $\frakX'_{\Gamma}$ are 
\begin{equation}
  \begin{split}
    tZ_{ij}Z_{kl} - Z_{ik}Z_{jl} + Z_{il}Z_{jk} = 0, 
      & \quad \text{$i,j \in I_{\pm}$ and $k,l \in I_{\mp}$},
    \\
    Z_{ij}Z_{kl} - Z_{ik}Z_{jl} + tZ_{il}Z_{jk} = 0, 
      & \quad \text{$i, l \in I_-$ and $j,k \in I_+$},
    \\
    Z_{ij}Z_{kl} - Z_{ik}Z_{jl} + Z_{il}Z_{jk} = 0, 
      & \quad \text{otherwise}. 
  \end{split}
  \label{eq:eq_X'_Gamma}
\end{equation}
Then it is easy to see the following.

\begin{lem}
  The family $\mathfrak{X}'_{\Gamma} \to \bC$ is 
  $(U_{d_+} \times U_{d_-})$-invariant.
  In particular, $\mathfrak{X}'_{\Gamma}$ admits 
  a natural $U_{d_{\beta}}$-action 
  if $U_{d_{\beta}} \subset U_{d_+} \times U_{d_-}$.
\end{lem}

We study the central fiber of $f'_{\Gamma}: \frakX'_{\Gamma} \to \bC$.
Let $P = P_+ \cup_d P_-$ be the subdivision of the reference polygon
by the diagonal $d$, where $P_{\pm}$ is a polygon whose sides are
$e_i$ ($i \in I_{\pm}$) and $d=d_{\alpha}$.
Each $P_{\pm}$ has a triangulation $\Gamma_{\pm}$ 
induced from $\Gamma$
as in Figure \ref{fg:subdivision}.
\begin{figure}
  \begin{center}
    \includegraphics*[bb= 0 0 119 91]{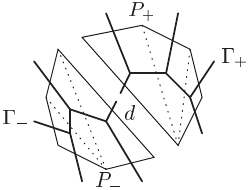}
  \end{center}
  \caption{A subdivision of a polygon and the induced trees}
  \label{fg:subdivision}
\end{figure}
We write elements of $\widetilde{\Gr}_{P_+}$ as
\[
  \begin{bmatrix}
    z_{i}^{P_+} & w_{i}^{P_+} \\
    z_{\alpha}^{P_+} & w_{\alpha}^{P_+}
  \end{bmatrix}_{i \in I_+}
  = 
  \begin{bmatrix}
    z_{p+1}^{P_+} & w_{p+1}^{P_+} \\ \vdots & \vdots \\
    z_{q}^{P_+} & w_{q}^{P_+} \\ z_{\alpha}^{P_+} & w_{\alpha}^{P_+}
  \end{bmatrix},
\]
and similarly for $\widetilde{\Gr}_{P_-}$.
We introduce two groups $S^1_0$ and $S^1_{d_{\alpha}}$
isomorphic to $S^1$,
and define an $S^1_0 \times S^1_{d_{\alpha}}$-action on
$\widetilde{\Gr}_{P_+} \times \widetilde{\Gr}_{P_-}$ by
\begin{equation}
\begin{split}
  &(s_0, s_{\alpha}) \cdot 
  \left( 
    \begin{bmatrix}
      z_{i}^{P_+} & w_{i}^{P_+} \\
      z_{\alpha}^{P_+} & w_{\alpha}^{P_+}
    \end{bmatrix}_{i \in I_+},
    \begin{bmatrix}
      z_{i}^{P_-} & w_{i}^{P_-} \\
      z_{\alpha}^{P_-} & w_{\alpha}^{P_-}
    \end{bmatrix}_{i \in I_-}
  \right) \\
  & \qquad \qquad =
  \left( 
    \begin{bmatrix}
      s_0 z_{i}^{P_+} & s_0 w_{i}^{P_+} \\
      s_{\alpha} z_{\alpha}^{P_+} & s_{\alpha}w_{\alpha}^{P_+}
    \end{bmatrix}_{i \in I_+},
    \begin{bmatrix}
      s_0 z_{i}^{P_-} & s_0 w_{i}^{P_-} \\
      s_{\alpha}^{-1} z_{\alpha}^{P_-} & s_{\alpha}^{-1} w_{\alpha}^{P_-}
    \end{bmatrix}_{i \in I_-}
  \right)
\end{split}
\label{eq:HMM-action}
\end{equation}
for $(s_0, s_{\alpha}) \in S^1_0 \times S^1_{d_{\alpha}}$.
Let $\bC^*_0 \times \bC^*_{d_{\alpha}}$ be the complexification of 
$S^1_0 \times S^1_{d_{\alpha}}$, which acts on 
$\widetilde{\Gr}_{P_+} \times \widetilde{\Gr}_{P_-}$ in an obvious way.

\begin{prop} \label{prop:central_fiber}
  The central fiber $X'_{\Gamma,0} = (f'_{\Gamma})^{-1}(0)$ 
  is isomorphic to the GIT quotient
  \[
    \widetilde{\Gr}_{P_+} \times \widetilde{\Gr}_{P_-}
    \GIT_{(2,0)} \bC^*_0 \times \bC^*_{d_{\alpha}},
  \]
  where the polarization is chosen in such a way that the weights of
  the actions of $\bC^*_0 \times \bC^*_{d_{\alpha}}$ is $(2,0)$.
  Moreover, the subfamily $\mathfrak{X}_{\Gamma}|_{t_{\alpha}=0}$ of
  $\mathfrak{X}_{\Gamma}$ 
  is induced from the degenerations $\mathfrak{X}_{\Gamma_{\pm}}$ of 
  $\Gr_{P_{\pm}}$ defined by $\Gamma_{\pm}:$
  \begin{equation}
    \mathfrak{X}_{\Gamma} |_{t_{\alpha}=0} \cong
    \widetilde{\mathfrak{X}}_{\Gamma_+} \times 
    \widetilde{\mathfrak{X}}_{\Gamma_-}
    \GIT_{(2,0)} \bC^*_0 \times \bC^*_{d_{\alpha}}.
    \label{eq:induced_degeneration}
  \end{equation}
\end{prop}

\begin{proof}
First note from \eqref{eq:eq_X'_Gamma}
that the defining 
equations for $X'_{\Gamma,0}$ are
\begin{equation}
  \begin{split}
    - Z_{ik}Z_{jl} + Z_{il}Z_{jk} = 0, 
      & \quad \text{$i,j \in I_{\pm}$ and $k,l \in I_{\mp}$},
    \\
    Z_{ij}Z_{kl} - Z_{ik}Z_{jl}  = 0, 
      & \quad \text{$i, l \in I_-$ and $j,k \in I_+$},
    \\
    Z_{ij}Z_{kl} - Z_{ik}Z_{jl} + Z_{il}Z_{jk} = 0, 
      & \quad \text{otherwise}. 
  \end{split}
  \label{eq:central_fiber}
\end{equation}
On the other hand,
the $\bC^*_0 \times \bC^*_{d_{\alpha}}$-action 
on the Pl\"ucker coordinates on $V_{P_{\pm}}$ is given by
\[
  Z_{ij}^{P_{\pm}} \mapsto s_0^2 Z_{ij}^{P_{\pm}},
  \quad
  Z_{i \alpha}^{P_{\pm}} \mapsto 
  s_0 s_{\alpha}^{\pm 1} Z_{i \alpha}^{P_{\pm}}
  \quad (i, j \in I_{\pm}),
\]
so that the ring 
$
 \bC [Z_{ij}^{P_+}, Z_{i \alpha}^{P_+},
  Z_{ij}^{P_-}, Z_{i \alpha}^{P_-}]^{\bC^*_{d_{\alpha}}}
$
of $\bC^*_{d_{\alpha}}$-invariants is generated by
$Z_{ij}^{P_{\pm}}$ ($i,j \in I_{\pm}$) and 
$Z_{i \alpha}^{P_+} Z_{j \alpha}^{P_-}$ ($i \in I_+$, $j \in I_-$).
Hence the map
$
 \bC[Z_{ij}] \to
 \bC [Z_{ij}^{P_+}, Z_{i \alpha}^{P_+},
  Z_{ij}^{P_-}, Z_{i \alpha}^{P_-}]
$
given by
  \begin{equation}
    Z_{ij} = \begin{cases}
      Z_{ij}^{P_{\pm}}, & i, j \in I_{\pm},\\
      Z_{i \alpha}^{P_+} Z_{j \alpha}^{P_-} ,
        & \text{$i \in I_+$ and $j \in I_-$}
    \end{cases}
  \label{eq:GIT_quot1}
\end{equation}
is a surjection to the invariant ring
$
 \bC [Z_{ij}^{P_+}, Z_{i \alpha}^{P_+},
  Z_{ij}^{P_-}, Z_{i \alpha}^{P_-}]^{\bC^*_{d_{\alpha}}},
$
and defines an embedding
\begin{equation} \label{eq:embedding}
  \widetilde{\Gr}_{P_+} \times \widetilde{\Gr}_{P_-} 
  \GIT_0 \bC^*_{d_{\alpha}}
  \hookrightarrow \textstyle{\bigwedge^2} \bC^n,
\end{equation}
of the GIT quotient.
It is easy to see that
the defining equations for the image
of the embedding \eqref{eq:embedding}
coincides with \eqref{eq:central_fiber},
so that
$
 \widetilde{\Gr}_{P_+} \times \widetilde{\Gr}_{P_-} 
  \GIT_0 \bC^*_{d_{\alpha}}
$
is isomorphic to the cone
$\widetilde{X}'_{\Gamma, 0} \subset \bigwedge^2 \bC^n$
over $X'_{\Gamma,0}$.
Under this identification, the $\mathbb{C}^*$-action defining the projection
$\widetilde{X}'_{\Gamma,0} \to X'_{\Gamma,0}$
coincides with the $\bC^*_0$-action, and hence we have
\[
  X'_{\Gamma,0} \cong \bigl(
  \widetilde{\Gr}_{P_+} \times \widetilde{\Gr}_{P_-} 
  \GIT_{0} \bC^*_{d_{\alpha}} \bigr)
  \GIT_{2} \bC^*_0.
\]
It is easy to see that the map \eqref{eq:GIT_quot1} extends
to an isomorphism
\[
  \bigl( \widetilde{\mathfrak{X}}_{\Gamma_+} \times 
  \widetilde{\mathfrak{X}}_{\Gamma_-} 
  \GIT_0 \bC^*_{d_{\alpha}} \bigr)
  \GIT_{2} \bC^*_0
  \longrightarrow \mathfrak{X}_{\Gamma}|_{t_{\alpha = 0}},
\]
and Proposition \ref{prop:central_fiber} is proved.
\end{proof}

\begin{rem}
From the proof, there is a $\mathbb{C}^*_0$-action on
$\widetilde{X}'_{\Gamma,0} \cong \widetilde{\Gr}_{P_+} \times \widetilde{\Gr}_{P_-} 
  \GIT_{0} \bC^*_{d_{\alpha}}$
such that the coordinates $Z_{ij}$ have weights 1.
$\widetilde{X}'_{\Gamma,0} \GIT_{2} \bC^*_0$ and
$\widetilde{X}'_{\Gamma,0} \GIT_{1} \bC^*_0$ are isomorphic as algebraic 
varieties (without polarizations). We will write
$(\widetilde{\Gr}_{P_+} \times \widetilde{\Gr}_{P_-}
    \GIT_{0} \bC^*_{d_{\alpha}}) \GIT_{1} \bC^*_0$ as 
$\widetilde{\Gr}_{P_+} \times \widetilde{\Gr}_{P_-}
    \GIT_{(1,0)} \bC^*_0 \times \bC^*_{d_{\alpha}}$
for simplicity.
\end{rem}

We now define a {\it degeneration of $\Gr(2,n)$ in stages} as follows.
For $\alpha = 1, \dots , n-3$, let
$f^{(\alpha)}_{\Gamma} : \frakX_{\Gamma}^{(\alpha)} \to \mathbb{C}$ 
be a subfamily of $\frakX_{\Gamma}$ defined by
\[
  t_1 = \dots = t_{\alpha-1} = 0,
  \quad
  t_{\alpha+1} = \dots = t_{n-3} = 1,
\]
and write its fibers as 
$X_{\Gamma, t}^{(\alpha)} = (f^{(\alpha)}_{\Gamma})^{-1}(t)$.
Then $\frakX^{(\alpha)}_{\Gamma}$ ($\alpha = 1, \dots, n-3$) give a sequence 
of families such that
$X_{\Gamma,1}^{(1)} = \Gr(2,n)$, 
$X_{\Gamma,0}^{(n-3)} = X_{\Gamma, 0}$, and
$X_{\Gamma,0}^{(\alpha)} = X_{\Gamma,1}^{(\alpha +1)}$
for each $\alpha$.
Furthermore, from the choice of the numbering of the diagonals, 
$\mathfrak{X}_{\Gamma}^{(\alpha)}$
admits actions of $U_{d_{\beta}}$ for all $\beta \ge \alpha$.
Let 
\[
  P = P^{(\alpha)}_1 \cup \dots \cup P^{(\alpha)}_{\alpha+1}
\]
be the subdivision of $P$ given by the diagonals
$d_1, \dots, d_{\alpha}$,
and $\Gamma^{(\alpha)}_m$ be the triangulation of
$P^{(\alpha)}_m$ induced from $\Gamma$.
For each diagonal $d_{\beta}$ we introduce $S^1_{d_{\beta}} \cong S^1$ and its 
complexification $\bC^*_{d_{\beta}}$,
and extend (\ref{eq:HMM-action}) to a 
$\bC^*_0 \times \bC^*_{d_1} \times \dots \times \bC^*_{d_{\alpha}}$-action on 
$\widetilde{\Gr}_{P^{(\alpha)}_1}\times \dots 
\times \widetilde{\Gr}_{P^{(\alpha)}_{\alpha+1}}$ in an obvious manner.
Namely, $\bC^*_0$ acts diagonally on the coordinates 
having a side $e_1, \dots, e_n$ of $P$ in its index, 
while $\bC^*_{d_{\beta}}$ acts anti-diagonally on the 
coordinates indexed the diagonal $d_{\beta}$.
Then Proposition \ref{prop:central_fiber} implies

\begin{cor} \label{cor:alpha-th}
  The central fiber $X_{\Gamma, 0}^{(\alpha)} \subset 
  \mathbb{P} (\bigwedge^2 \mathbb{C}^n)$ of the $\alpha$-th stage
  $\mathfrak{X}_{\Gamma}^{(\alpha)}$ is isomorphic to 
  \[
    \widetilde{\Gr}_{P^{(\alpha)}_1} \times 
    \dots \times \widetilde{\Gr}_{P^{(\alpha)}_{\alpha+1}}
   \GIT_{(1,0,\dots,0)} 
   \bC^*_0 \times \bC^*_{d_1} \times \dots \times \bC^*_{d_\alpha} ,
  \]
  where the weight of the 
  $\bC^*_0 \times \bC^*_{d_1} \times \dots \times \bC^*_{d_\alpha}$-action 
  is $(1,0, \dots, 0)$.
  Moreover, the $(\alpha+1)$-st stage of the degeneration is given by
  \[
    \mathfrak{X}^{(\alpha +1)}_{\Gamma} \cong
    \widetilde{\mathfrak{X}}^{(1)}_{\Gamma^{(\alpha)}_1} \times
    \widetilde{\Gr}_{P^{(\alpha)}_2} \times 
    \dots \times \widetilde{\Gr}_{P^{(\alpha)}_{\alpha+1}}
   \GIT_{(1,0,\dots,0)} 
   \bC^*_0 \times \bC^*_{d_1} \times \dots \times \bC^*_{d_\alpha} ,
  \]
  where we assume that $d_{\alpha+1}$ is a diagonal of $P^{(\alpha)}_1$,
  and $\mathfrak{X}^{(1)}_{\Gamma^{(\alpha)}_1}$ is the first stage of the degeneration 
  of $\Gr_{P^{(\alpha)}_1}$ corresponding to $d_{\alpha+1}$.
\end{cor}

Now let us look at the last stage $\alpha = n-3$ of the degeneration.
The reference polygon is divided as
$$
 P = P_1^{(n-3)} \cup \cdots \cup P_{n-2}^{(n-3)}
$$
where all $P^{(n-3)}_m$ $(1 \le m \le n-2)$ are triangles,
and we have 
$$
 \widetilde{\Gr}_{P^{(n-3)}_m}
  = V_{P^{(n-3)}_m}
  = \Spec \bC \left[Z_{ab}^{P_m^{(n-3)}}\right]
  \cong \textstyle{\bigwedge^2} \mathbb{C}^3,
$$
where $a, b \in \{ e_1, \dots, e_n, d_1, \dots, d_{n-3} \}$
are edges in the triangle $P_m^{(n-3)}$.

\begin{cor}[Howard, Manon and Millson \cite{Howard-Manon-Millson}]
 \label{cr:XG}
  The central fiber $\XG := X_{\Gamma, 0}$ of 
  $\mathfrak{X}_{\Gamma}$ is a toric variety given by
  \begin{multline*}
    V_{P^{(n-3)}_1}
    \times \dots \times V_{P^{(n-3)}_{n-2}}
   \GIT_{(1,0,\dots,0)} 
   \bC^*_0\times \bC^*_{d_1} \times \dots 
   \times \bC^*_{d_{n-3}} \\
   \cong
   (\textstyle{\bigwedge^2} \mathbb{C}^3 )^{n-2} \GIT (\bC^*)^{n-2}.
  \end{multline*}
\end{cor}

The cone
\begin{multline*}
 V_{P^{(n-3)}_1}
  \times \dots \times V_{P^{(n-3)}_{n-2}}
  \GIT_{(0,\dots,0)} 
   \bC^*_{d_1} \times \dots \times \bC^*_{d_{n-3}} \\
 =
  \Spec \left(
  \bC \left[ Z_{ab}^{P_m^{(n-3)}} \right]^{
   \bC^*_{d_1} \times \dots \times \bC^*_{d_{n-3}}}
   \right)
\end{multline*}
over $\XG$
is embedded into the affine space $\bigwedge^2 \bC^n$
by the surjection
$$
 \bC[Z_{ij}] \to \bC \left[ Z_{ab}^{P_m^{(n-3)}} \right]^{
   \bC^*_{d_1} \times \dots \times \bC^*_{d_{n-3}}},
  \qquad
 Z_{ij} \mapsto \prod_{((a,b),m)} Z_{ab}^{P_m^{(n-3)}},
$$
where the product on the right hand side
runs over triples $((a, b), m)$ of edges $(a, b)$
of a triangle $P_m^{(n-3)}$
in the triangulation
$P = P_1^{(n-3)} \cup \dots \cup P_{n-2}^{(n-3)}$
crossed by the path $\gamma(i, j)$
connecting the $i$-th and the $j$-th leaves
as in Figure \ref{fig:graphpath}.

Defining equations for the image of the embedding
$
 \XG \hookrightarrow \bP \lb \bigwedge^2 \bC^n \rb
$
come from those of the embedding
$$
 V_{P_a} \times V_{P_b}
  \bigGIT_0 \bC_{d_\alpha}^*
  \hookrightarrow \bigwedge^2 \bC^4,
$$
where $P_a = P_a^{(n-3)}$ and
$P_b = P_b^{(n-3)}$ are two triangles
sharing a diagonal $d = d_{\alpha}$. 
Let $a_1$ and $a_2$ (resp. $b_1$ and $b_2$) be
the remaining  edges of $P_a$ (resp. $P_b$).
Let $P_c = P_c^{(n-2)} = P_a \cup P_b$ be a quadrilateral
obtained as the union of $P_a$ and $P_b$.
Then the inclusion is defined by
the homomorphism
\begin{align*}
 Z_{a_1 b_1}^{P_c} &= Z_{a_1 d}^{P_a} Z_{b_1 d}^{P_b}, \\
 Z_{a_1 b_2}^{P_c} &= Z_{a_1 d}^{P_a} Z_{b_2 d}^{P_b}, \\
 Z_{a_2 b_1}^{P_c} &= Z_{a_2 d}^{P_a} Z_{b_1 d}^{P_b}, \\
 Z_{a_2 b_2}^{P_c} &= Z_{a_2 d}^{P_a} Z_{b_2 d}^{P_b}, \\
 Z_{a_1 a_2}^{P_c} &= Z_{a_1 a_2}^{P_a}, \\
 Z_{b_1 b_2}^{P_c} &= Z_{b_1 b_2}^{P_b},
\end{align*}
so that the defining equation
of the image is given by
$$
 Z_{a_1 b_1}^{P_c} Z_{a_2 b_2}^{P_c}
  = Z_{a_1 b_2}^{P_c} Z_{a_2 b_1}^{P_c}.
$$
It follows that
the singular locus of
$
 V_{P_a} \times V_{P_b}
  \bigGIT_0 \bC_{d_\alpha}^*
$
is given by
\begin{equation*} 
 Z_{a_1 b_1}^{P_c} = Z_{a_1 b_2}^{P_c}
  = Z_{a_2 b_1}^{P_c} = Z_{a_2 b_2}^{P_c} = 0,
\end{equation*}
that is,
\begin{equation} \label{eq:singularity}
 Z_{a_1 d}^{P_a} = Z_{a_2 d}^{P_a} = 0 \, \text{ or } \,
 Z_{d b_1}^{P_b} = Z_{d b_2}^{P_b} = 0,
\end{equation}
with $Z_{a_1 a_2}^{P_a}$ and $Z_{b_1 b_2}^{P_b}$ arbitrary.
This gives a codimension 3 singularity
in $\XG$.
Since $\XG$ is a toric variety
obtained as the quotient of an affine space
by torus action,
a singular point of $\XG$ comes from
a point on the affine space
$
 V_{P^{(n-3)}_1} \times \dots \times V_{P^{(n-3)}_{n-2}}
$
where the torus action has a non-trivial stabilizer.
Such a point is contained in one of 
the locus defined by \eqref{eq:singularity},
and one obtains the following:

\begin{proposition} \label{pr:Sing_XG}
The singular locus $\Sing(\XG)$ of $\XG$ is the union
$$
 \Sing (\XG)_A 
  = \bigcup_{\alpha = 1}^{n-3} \Sing (\XG)_\alpha
$$
where
$$
 \Sing (\XG)_{\alpha} 
 = \{ Z_{a_1 d_{\alpha}}^{P_a} = Z_{a_2 d_{\alpha}}^{P_a} = 0 \, 
   \text{ or } \, 
      Z_{b_1 d_{\alpha}}^{P_b} = Z_{b_2 d_{\alpha}}^{P_b} = 0 \}.
$$
\end{proposition}

\section{Toric degeneration of the integrable system}
 \label{sc:deg_CIS}

The following definition is introduced
in \cite[Definition 1.1]{Nishinou-Nohara-Ueda_TDGCSPF}:

\begin{df} \label{def:toric_degeneration}
Let $\Phi : X \to \bR^N$ be a completely integrable system
on a projective K\"ahler manifold $(X, \omega)$.
A {\em toric degeneration} of $\Phi$
consists of a flat family $f : \frakX \to B$
of algebraic varieties over a complex manifold $B$,
a K\"{a}hler form $\omegatilde$ on the smooth locus of $\frakX$,
a piecewise smooth path $\gamma : [0,1] \to B$,
a continuous map
$\Phitilde : \frakX |_{\gamma([0,1])} \to \bR^N$
on  $\frakX |_{\gamma([0,1])} = f^{-1}(\gamma([0,1]))$,
and a flow $\phi_t$ on $\frakX |_{\gamma([0,1])}$ 
which covers the path $\gamma$ and is defined away from the union 
$\bigcup_{t \in [0,1]} \mathrm{Sing}(X_t)$ 
of the singular loci of the fibers $X_t = f^{-1}(\gamma(t))$
such that
\begin{itemize}
 \item for each $t \in [0,1]$, the restriction 
       $\Phi_t = \Phitilde|_{X_t}$ is a completely integrable system
       on the K\"{a}hler variety 
       $(X_t, \omega_t = \omegatilde|_{X_t})$,
       whose image $\Phi_t(X_t)$ is a convex polytope $\Delta$
       independent of $t$,
 \item $(X_1, \omega_1)$ is isomorphic to $(X, \omega)$
       as a K\"{a}hler manifold,
 \item $\Phi_1$ coincides with $\Phi$ 
       under the above isomorphism $X_1 \cong X$,
 \item $(X_0, \omega_0)$ is a toric variety
       with a torus-invariant K\"{a}hler form, 
 \item $\Phi_0 : X_0 \to \bR^N$ is the moment map
       for the torus action on $X_0$ (hence $\Delta$ is 
       a moment polytope of $X_0$), and
 \item For each $t$, there is an open dense subset $X_t^{\circ} \subset X_t$ 
       such that
       the flow $\phi_t$ sends $X_{t'}^{\circ}$ to another fiber 
       $X_{t'-t}^{\circ}$ preserving the symplectic structures 
       and the completely integrable systems:
       \[
         \xymatrix{
           (X_{t'}^{\circ}, \omega_{t'})  \ar[dr]_{\Phi_{t'}} 
           \ar[rr]^{ \phi_{t}}
           & & (X_{t'-t}^{\circ}, \omega_{t'-t}) 
           \ar[dl]^{\Phi_{t'-t}} \\ 
           & \Delta &
         }
       \]
\end{itemize}
\end{df}

In this section we construct a toric degeneration of 
$\Psi_{\Gamma} : \Gr(2,n) \to \mathbb{R}^{\N}$.
We consider the family
$f_{\Gamma}: \frakX_{\Gamma} \to \mathbb{C}^{n-3}$ constructed in the
previous section,
and let $\gamma$ be a piecewise linear path
connecting the vertices
\[
  (1, \dots, 1), (0,1,\dots,1), \dots , (0,\dots,0,1),
  (0,\dots,0).
\]
Then the restriction of $\frakX_{\Gamma}$ to $\gamma$ is the degeneration 
in stages.
We take a K\"ahler form $\widetilde{\omega}_{\bld{r}}$ on
$\mathfrak{X}_{\Gamma}$ such that the restriction 
$\omega_{\bld{r}, t} = \widetilde{\omega}_{\bld{r}}|_{X_{\Gamma, t}}$ 
to each fiber of $f_{\Gamma}$
is the constant multiple 
$|\bld{r}| \cdot \omega_{\mathrm{FS}}|_{X_{\Gamma, t}}$
of the Fubini-Study form $\omega_{\mathrm{FS}}$ on 
$\mathbb{P}(\bigwedge^2 \mathbb{C}^n)$.
For each stage $f^{(\alpha)}_{\Gamma}: \frakX^{(\alpha)}_{\Gamma} \to \mathbb{C}$
of the degeneration, we define a map
\[
  \widetilde{\Psi}^{(\alpha)}_{\Gamma} : 
  \frakX^{(\alpha)}_{\Gamma} \longrightarrow \bR^\N
\]
as follows.
Recall that $P^{(\alpha-1)}_1, \dots, P^{(\alpha-1)}_{\alpha}$ are subpolygons
obtained by cutting $P$ along the diagonals $d_1, \dots, d_{\alpha-1}$, and 
$\Gamma^{(\alpha-1)}_m$ are the triangulation of $P^{(\alpha-1)}_m$ 
induced from $\Gamma$.
From Corollary \ref{cor:alpha-th}, we have
\[
    \mathfrak{X}^{(\alpha)}_{\Gamma} \cong
    \widetilde{\mathfrak{X}}^{(1)}_{\Gamma^{(\alpha-1)}_1} \times
    \widetilde{\Gr}_{P^{(\alpha-1)}_2} \times 
    \dots \times \widetilde{\Gr}_{P^{(\alpha-1)}_{\alpha}}
   \GIT 
   \bC^*_0 \times \bC^*_{d_1} \times \dots \times \bC^*_{d_{\alpha-1}} ,
\]
where we assume that $d_{\alpha}$ is a diagonal in $P^{(\alpha-1)}_1$.
The actions of $U_{d_{\beta}}$ for $\beta \ge \alpha$ and 
$S^1_{e_i}$ are induced from those on
$\widetilde{\mathfrak{X}}^{(1)}_{\Gamma^{(\alpha-1)}_1}$ or 
$\widetilde{\Gr}_{P^{(\alpha-1)}_m}$ for some $m$.
We define 
\[  
  \widetilde{\lambda}_{\beta,j} : \frakX^{(\alpha)}_{\Gamma} 
  \longrightarrow \mathbb{R},
  \quad
  j=1, 2
\]
to be the first and the second eigenvalues
of the values of the moment map
of the $U_{d_{\beta}}$-action,
which is a natural extension of $\lambda_{\beta,j}$.
We also extend the moment map $\psi_{e_i}$ of the $S^1_i$-action  
to
\[
  \widetilde{\psi}_{e_i} : \frakX^{(\alpha)}_{\Gamma} 
  \longrightarrow \mathbb{R},
\]
for $i=1, \dots, n$. 
The space $\frakX^{(\alpha)}_{\Gamma}$ has
an action of another torus 
$S^1_{d_1} \times \dots \times S^1_{d_{\alpha}}$,
where $S^1_{d_{\gamma}}$ ($1 \le \gamma \le \alpha$)
acts diagonally on Pl\"ucker coordinates 
$Z_{a d_{\gamma}}^{P^{(\alpha-1)}_m}$ having $d_{\gamma}$ in their indices.
In other words, the weight of a (Pl\"ucker) coordinate $Z_{ij}$ of 
$\mathbb{P} (\bigwedge^2 \mathbb{C}^n)$ with respect to the 
$S^1_{d_{\gamma}}$-action
is 1 if the path $\gamma (i,j)$ from $i$ to $j$ crosses
$d_{\gamma}$, and 0 otherwise.
Let
\[
  \widetilde{\mu}_{S^1_{\gamma}} : \frakX^{(\alpha)}_{\Gamma} 
  \longrightarrow \mathbb{R}
\]
be the moment map of the $S^1_{d_{\gamma}}$-action, and set
\[
  \widetilde{\nu}_{\gamma} = 
  \frac 12 \left( \sum_{i \in I_{\gamma}} \psi_{e_i} 
  - \widetilde{\mu}_{S^1_{\gamma}} \right).
\]
We define
$
 \widetilde{\Psi}^{(\alpha)}_{\Gamma} : \frakX^{(\alpha)}_{\Gamma} 
  \to \bR^\N
$
by
$$
  \widetilde{\Psi}^{(\alpha)}_{\Gamma} =
  \bigl( \widetilde{\nu}_{1}, \dots, \widetilde{\nu}_{\alpha-1},
         \widetilde{\lambda}_{\alpha,2}, \dots, \widetilde{\lambda}_{n-3,2},
         \widetilde{\psi}_{e_1}, \dots,  \widetilde{\psi}_{e_{n-1}}
  \bigr),
$$
and $\Psi^{(\alpha)}_{\Gamma, t} 
= \widetilde{\Psi}^{(\alpha)}_{\Gamma}|_{X^{(\alpha)}_{\Gamma, t}}$.
Then $\Psi^{(1)}_{\Gamma, 1} = \Psi_{\Gamma}$ on 
$X^{(1)}_{\Gamma, 1} = \Gr(2,n)$ from the construction.

We define $\phi^{(\alpha)}_t$ on $\mathfrak{X}^{(\alpha)}_{\Gamma}$ 
to be the
{\it gradient-Hamiltonian flow} of $f_{\Gamma}^{(\alpha)}$ 
introduced by W.-D. Ruan \cite{Ruan_I}.
Regarding
$
 f^{(\alpha)}_\Gamma
  : \frakX^{(\alpha)}_\Gamma \to \bC
$
as a holomorphic function, the normalized gradient-Hamiltonian vector field 
is defined by
\[
  \xi^{(\alpha)} := -\frac {\nabla (\Re f^{(\alpha)}_{\Gamma})}
                  {|\nabla (\Re f^{(\alpha)}_{\Gamma})|^2}
        = \frac {\xi_{\Im f^{(\alpha)}_{\Gamma}}}
                  {|\xi_{\Im f^{(\alpha)}_{\Gamma}}|^2},
\]
where $\nabla (\Re f^{(\alpha)}_{\Gamma})$ is the gradient vector field of the 
real part of $f^{(\alpha)}_{\Gamma}$, and
$\xi_{\Im f^{(\alpha)}_{\Gamma}}$ is the Hamiltonian vector field of
the imaginary part of $f^{(\alpha)}_{\Gamma}$.
It is shown in \cite{Ruan_I} that $\xi^{(\alpha)}$ is defined on 
the smooth locus of fibers $X^{(\alpha)}_{\Gamma,t}$, and its flow
$\phi^{(\alpha)}_t = \exp t \xi^{(\alpha)}$ gives a symplectomorphism
\[
   \phi^{(\alpha)}_{1-t} : 
   (W^{(\alpha)}_{\Gamma,1}, \omega_{\bld{r},1})
   \longrightarrow
   (W^{(\alpha)}_{\Gamma,t}, \omega_{\bld{r},t})
\]
for some open subsets 
$W^{(\alpha)}_{\Gamma,t} \subset X^{(\alpha)}_{\Gamma,t}$.

\begin{rem}
  The authors do not know whether $\phi^{(\alpha)}_{1-t}$ can be 
  extended to $X^{(\alpha)}_{\Gamma,1} \to X^{(\alpha)}_{\Gamma,t}$.
  Note that the total space of the family is not smooth in general, 
  and hence we can not apply the analysis in \cite{Ruan_II}.
\end{rem}

We begin with the proof of Theorem \ref{th:toricdeg_nohara} 
for the first stage.

\begin{lem} \label{lem:ham-1st_stage}
  $\Psi^{(1)}_{\Gamma,t} : X_{\Gamma,t}^{(1)} \to \bR^\N$
  is a completely integrable system for each $t$.
  The gradient-Hamiltonian flow $\phi^{(1)}_t$ is defined on an open dense 
  subset $W_{\Gamma, t}^{(1)} \subset X_{\Gamma, t}^{(1)}$ and 
  gives a symplectomorphism preserving the completely integrable systems,
  i.e., 
  \[
    \xymatrix{
    (W^{(1)}_{\Gamma,t}, \omega_{\bld{r}, t})  
    \ar[dr]_{\Psi^{(1)}_{\Gamma, t}} 
    \ar[rr]^{ \phi^{(1)}_{t-t'}}
    & & (W_{\Gamma, t'}^{(1)}, \omega_{\bld{r}, t'}) 
        \ar[dl]^{\Psi^{(1)}_{\Gamma, t'}} \\ 
    & \bR^\N &
            }
  \]
  commutes.
  Furthermore $\Psi^{(1)}_{\Gamma,0}$ coincides with $\Psi^{(2)}_{\Gamma,1}$
  on $X^{(1)}_{\Gamma,0} = X^{(2)}_{\Gamma,1}$.
\end{lem}

\begin{proof}
Poisson commutativity follows from the same argument as in the proof
of Proposition \ref{prop:Poisson_commutative}.
The fact that $f^{(1)}_{\Gamma}$ is invariant under the actions of
$U_{d_{\beta}}$ and $S^1_{e_i}$ implies that $\phi_t^{(1)}$
preserves the completely integrable system 
just as in \cite[Section 7]{Nishinou-Nohara-Ueda_TDGCSPF}.
Since $X_{\Gamma, t}^{(1)}$ is smooth for $t \ne 0$, 
$\phi_{1-t}^{(1)}$ is defined on $X_{\Gamma,1}^{(1)}$, and
\[
  \phi_{1-t}^{(1)} : X_{\Gamma, 1}^{(1)} \longrightarrow X_{\Gamma,t}^{(1)}
\]
is a symplectomorphism satisfying 
$\Psi_{\Gamma, 1}^{(1)} = {\phi_{1-t}^{(1)}}^* \Psi_{\Gamma, t}^{(1)}$.
Hence the functions in $\Psi_{\Gamma, t}^{(1)}$ are
functionally independent
on an open dense subset of $X_{\Gamma, t}^{(1)}$.

For $t=0$, there is an open subset 
$W_{\Gamma,1}^{(1)} \subset X_{\Gamma,1}^{(1)}$
on which $\phi_1^{(1)}$ is defined and gives a symplectomorphism
\begin{equation}
  \phi_{1}^{(1)} : W_{\Gamma, 1}^{(1)} 
  \longrightarrow X_{\Gamma,0}^{(1)} 
  \setminus \mathrm{Sing}(X_{\Gamma,0}^{(1)}).
  \label{eq:flow_to_X0}
\end{equation}
From 
$\Psi_{\Gamma, 1}^{(1)} = {\phi_{1}^{(1)}}^* \Psi_{\Gamma, 0}^{(1)}$
and the fact that $\Psi_{\Gamma, 1}^{(1)}$ is functionally independent,
$\Psi_{\Gamma, 0}^{(1)}$ is also functionally independent.
For each $t$,
set $W_{\Gamma, t}^{(1)} = \phi_{1-t}^{(1)} (W_{\Gamma, 1}^{(1)})$.
To show that $W_{\Gamma, t}^{(1)}$ is dense in $X_{\Gamma, t}^{(1)}$,
we consider the inverse image 
$\mathrm{Int}(X_{\Gamma, t}^{(1)}) 
= \bigl(\Psi_{\Gamma, t}^{(1)}\bigr)^{-1} 
 (\mathrm{Int}(\Delta_{\Gamma}))$ 
of the interior of $\Delta_{\Gamma}$.
Note that 
\[
  \mathrm{Int}(X_{\Gamma, t}^{(1)}) 
  \cong T^{\N} \times \mathrm{Int}(\Delta_{\Gamma})
\]
is a dense subset in $X_{\Gamma, t}^{(1)}$ on which 
$\Psi_{\Gamma,t}^{(1)}$ defines a free $T^{\N}$-action.
We also note that $\mathrm{Int}(X_{\Gamma, 0}^{(1)})$
is contained in the smooth locus of $X_{\Gamma, 0}^{(1)}$.
Since $\phi^{(1)}_t$ preserves the Hamiltonian torus actions,
we have a a symplectomorphism
$\phi_t^{(1)} : \mathrm{Int}(X_{\Gamma, t}^{(1)}) \to 
\mathrm{Int}(X_{\Gamma, 0}^{(1)})$.
Hence we have
$\mathrm{Int}(X_{\Gamma, t}^{(1)}) \subset W_{\Gamma, t}^{(1)}$,
which shows that $W_{\Gamma, t}^{(1)}$ is dense in 
$X_{\Gamma, t}^{(1)}$.

Finally we show that $\widetilde{\lambda}_{1,2}$ coincides with 
$\widetilde{\nu}_1$ on $X^{(1)}_{\Gamma,0} = X^{(2)}_{\Gamma,1}$.
Let $P= P^{(1)}_1 \cup_{d_1} P^{(1)}_2$ be the subdivision of $P$ 
given by $d_1 = \sum_{i\in I_1} e_i$ such that
$P^{(1)}_2$ is a triangle. 
Then $X_{\Gamma,0}^{(1)}$ is written as
\[
  X_{\Gamma,0}^{(1)} \cong \widetilde{\Gr}_{P^{(1)}_1} \times V_{P^{(1)}_2}
  \GIT_{(1,0)} \bC^*_0 \times \bC^*_{d_1}.
\]
From the construction, 
$\widetilde{\lambda}_{1,2}, \dots, \widetilde{\lambda}_{n-3,2}$ and
$\widetilde{\psi}_{e_i}$  ($i \in I_1$) are induced from the completely 
integrable system $\Psi_{\Gamma^{(1)}_1}$ on $\Gr_{P^{(1)}_1}$,
while $\widetilde{\psi}_{e_i}$  ($i \not\in I_1$) are induced from a natural
$T^2$-action on $V_{P^{(1)}_2} \cong \bC^3$.
Note that the moment map of the $U(n)$-action on 
$\bP(\bigwedge^2 \bC^n)$ is given by
\[
  \mu_{U(n)} ([Z_{ij}]) = \frac{| \bld r|}{ \sum_{i,j} |Z_{ij}|^2}
  \left( \sum_k Z_{ik} \overline{Z_{jk}} \right)_{i,j = 1, \dots, n}.
\]
Recall that the coordinates $Z_{ij}$ on the central fiber 
$X_{\Gamma, 0}^{(1)}$
satisfy
\[
  Z_{ij} = \begin{cases}
    Z'_{ij}, & i, j \in I_1, \\
    Z'_{i d_1} Z''_{ d_1 j}, 
      & i \in I_1, j \notin I_1,\\
    Z''_{ij}, & i, j \notin I_1,
  \end{cases}
\]
where we write $Z'_{ij} := Z_{ij}^{P_1^{(1)}}$, 
$Z''_{ij} := Z_{ij}^{P_2^{(1)}}$ for simplicity.
Then the moment map of the $U_{d_1}$-action on $X_{\Gamma, 0}^{(1)}$
is given by
\begin{align*}
  \mu_{U_{d_1}} 
  &= \frac{| \bld r|}{ \sum_{i,j} |Z_{ij}|^2}
  \left( \sum_k Z_{ik} \overline{Z_{jk}} \right)_{i,j \in I_1} \\
  &= \frac{| \bld r|}{ \sum_{i,j} |Z_{ij}|^2}
  \left( \sum_{k \in I_1} Z'_{ik} \overline{Z'_{jk}} \right)_{i,j} 
  + \frac{| \bld r| \sum_{k \notin I_{d_1}} |Z''_{d_1 k}|^2}{ \sum_{i,j} |Z_{ij}|^2} 
  \left( Z'_{i d_1} \overline{Z'_{j d_1}} \right)_{i,j},
\end{align*}
which has eigenvalues
\begin{align*}
  \widetilde{\lambda}_{1,1} 
  &= \frac{| \bld r|}{ \sum_{i,j} |Z_{ij}|^2}
    \left( \sum_{i,j \in I_1} |Z_{ij}|^2 
  + \sum_{i \in I_1, \  j \notin I_1} |Z_{ij}|^2 
  \right), \\
  \widetilde{\lambda}_{1,2} 
  &= \frac{| \bld r|}{ \sum_{i,j} |Z_{ij}|^2}
    \sum_{i,j \in I_1} |Z_{ij}|^2 .
\end{align*}
Since 
$\widetilde{\lambda}_{1,1} + \widetilde{\lambda}_{1,2}
= \tr \mu_{U_{d_1}} = \sum_{i \in I_1} \widetilde\psi_{e_i}$
and 
\[
  \widetilde{\mu}_{S^1_{d_1}}
  = \frac{| \bld r|}{ \sum_{i,j} |Z_{ij}|^2}
  \sum_{i \in I_1, \  j \notin I_1} |Z_{ij}|^2 ,
\]
we have
\[
  \widetilde{\lambda}_{1,2} = \frac 12
  \left( \sum_{i \in I_1} \widetilde\psi_{e_i} 
  - \widetilde{\mu}_{S^1_{d_1}} \right)
  = \widetilde\nu_1
\]
on $X_{\Gamma, 0}^{(1)}$.
\end{proof}

Next proposition completes the proof of Theorem \ref{th:toricdeg_nohara}.

\begin{prop}
  For each $t$, 
  $\Psi^{(\alpha)}_{\Gamma,t} : X_{\Gamma,t}^{(\alpha)} \to \bR^\N$
  is a completely integrable system.
  The gradient-Hamiltonian flow $\phi^{(\alpha)}_t$ is defined on an 
  open dense subset $W^{(\alpha)}_{\Gamma, t} \subset X^{(\alpha)}_{\Gamma, t}$ 
  and preserves the completely integrable systems.
  Furthermore, $\Psi^{(\alpha)}_{\Gamma,0}$ coincides with $\Psi^{(\alpha+1)}_{\Gamma,1}$
  on $X^{(\alpha)}_{\Gamma,0} = X^{(\alpha+1)}_{\Gamma,1}$.
\end{prop}

\begin{proof}
We prove the proposition by induction on $\alpha$.
The case $\alpha = 1$ is proved
in Lemma \ref{lem:ham-1st_stage}.
Assume that the statement of the proposition holds
for $\Psi_{\Gamma, t}^{(\alpha-1)}$.
The same argument as in the proof of Lemma \ref{lem:ham-1st_stage}
shows that
there exists an open subset 
$W_{\Gamma, t}^{(\alpha)} \subset X_{\Gamma, t}^{(\alpha)}$ on which 
the gradient-Hamiltonian flow is defined and 
$\phi^{(\alpha)}_{t-t'} : W_{\Gamma, t}^{(\alpha)} \to W_{\Gamma, t'}^{(\alpha)}$
is a symplectomorphism preserving $\Psi^{(\alpha)}_{\Gamma,t}$.
From the hypothesis of induction, $\Phi_{\Gamma, 1}^{(\alpha)}$ is a completely 
integrable system, and hence 
$\Phi_{\Gamma, t}^{(\alpha)} 
= {\bigl( \phi^{(\alpha)}_{1-t} \bigr)^{-1}}^* \Phi_{\Gamma, 1}^{(\alpha)}$
is also a completely integrable system on $W^{(\alpha)}_{\Gamma, t}$.
We need to show is that $W_{\Gamma, t}^{(\alpha)}$ is dense in 
$X_{\Gamma, t}^{(\alpha)}$.

Let $P = P^{(\alpha-1)}_1 \cup \dots \cup P^{(\alpha-1)}_{\alpha}$ be 
the subdivision given by $d_1, \dots, d_{\alpha -1}$ and
suppose that $d_{\alpha}$ is a diagonal of $P^{(\alpha -1)}_1$.
Then we have
\[
    X_{\Gamma, t}^{(\alpha)} = 
    \widetilde{X}_{\Gamma^{(\alpha-1)}_1, t}^{(1)} \times 
    \widetilde{\Gr}_{P^{(\alpha -1)}_2} \times \dots \times 
    \widetilde{\Gr}_{P^{(\alpha -1)}_{\alpha}}
    \GIT
    \bC^*_0 \times \bC^*_{d_1} \times \dots \times \bC^*_{d_{\alpha -1}}.
\]
The functions $\widetilde{\lambda}_{\beta,j}$ ($\beta \ge \alpha$) and
$\widetilde{\psi}_{e_i}$ ($i = 1, \dots, n-1$) in $\Psi^{(\alpha)}_{\Gamma,t}$ are
induced from $\Psi^{(1)}_{\Gamma^{(\alpha-1)}_1, t}$ on 
$X_{\Gamma^{(\alpha-1)}_1, t}^{(1)}$ and 
$\Psi_{\Gamma^{(\alpha -1)}_m}$ on  $\Gr_{P^{(\alpha -1)}_m}$ ($m \ge 2$).
From Lemma \ref{lem:ham-1st_stage} and the fact that
the Hamiltonian torus action of $\Psi^{(\alpha)}_{\Gamma,t}$ and 
$\Psi_{\Gamma^{(\alpha -1)}_m}$ is defined on an open dense subset of 
$\Gr_{P^{(\alpha -1)}_m}$, 
the Hamiltonian actions of 
$\widetilde{\lambda}_{\beta,j}$ and $\widetilde{\psi}_{e_i}$
are also defined on an open subset of $X_{\Gamma, t}^{(\alpha)}$.
On the other hand, the Hamiltonian action of $\widetilde{\nu}_{\beta}$ 
is the diagonal $S^1_{d_{\beta}}$-action,
which is defined everywhere on $X_{\Gamma, t}^{(\alpha)}$ and 
transverse to the Hamiltonian actions of $\widetilde{\lambda}_{\beta,j}$
and $\widetilde{\mu}_{S^1_i}$.
In particular, $\mathrm{Int}(X_{\Gamma, t}^{(\alpha)})$, on which 
the Hamiltonian $T^{\N}$-action is free, is dense in $X_{\Gamma, t}^{(\alpha)}$,
and hence so is $W_{\Gamma, t}^{(\alpha)} \subset X_{\Gamma, t}^{(\alpha)}$.

Since
$\Psi_{\Gamma, 0}^{(\alpha)}$ is induced from the integrable systems 
on $X_{\Gamma^{(\alpha-1)}_1, 0}^{(1)}$ and $\Gr_{P^{(\alpha -1)}_m}$,
the last statement of the proposition follows from Lemma \ref{lem:ham-1st_stage}.
\end{proof}

Since the toric degeneration of
$\Psi_{\Gamma} : \Gr(2,n) \to \bR^\N$
is invariant under the action of maximal torus 
$T_{U(n)} = \prod_{i=1}^n S^1_{e_i}$,
we have the following:

\begin{cor}
  The toric degeneration 
  $(\mathfrak{X}_{\Gamma}^{(\alpha)}, \widetilde{\Psi}_{\Gamma}^{(\alpha)},
  \phi^{(\alpha)})$ of $\Psi_{\Gamma}$ induces a toric degeneration
  of the bending system on $\pol$ associated to $\Gamma$.
  In particular, $\Delta_{\Gamma}(\bld{r}) = \Phi_{\Gamma}(\pol)$ is 
  a moment polytope of the central fiber $X_{\Gamma,0} \GIT T_{U(n)}$.
\end{cor}

\begin{expl}
Let $n=5$ and assume that
the side lengths $r_1, \dots, r_5$ are close to each other
as in Example \ref{expl:pentagon}.
Then $\pol$ is isomorphic to $\mathbb{CP}^2$ blown-up
at four points in general position.
It follows from Figure \ref{fig:heptagon}
that the central fiber $X_{\Gamma, 0} \GIT T_{U(n)}$
is $\mathbb{CP}^2$ blown-up four times
at two pairs of infinitely-near points.
\end{expl}

\section{Properties of $\XG$}
 \label{sc:X_Gamma}

Let $\XG := X_{\Gamma, 0}$ be the toric variety
obtained as the central fiber
of the toric degeneration of $\Gr(2,n)$
associated with a triangulation $\Gamma$
of the reference polygon $P$.

\begin{lem}
  The torus fixed point set in the toric variety
  $\XG \subset \bP(\bigwedge^2 \bC^n)$
  consists of points $p_{kl}=[Z_{ij}]_{i,j} \in \bP(\bigwedge^2 \bC^n)$
  $(\{k,l\} \subset \{1, \dots, n\} )$
  defined by $Z_{ij}=0$   for all $\{i,j\} \ne \{k,l\}$.
  In particular, the number of fixed points is 
  $n(n-1)/2$.
\end{lem}

\begin{proof}
First note that any torus fixed point
in a toric variety is obtained
as an intersection of toric divisors.
It follows from the description of $\XG$
given in Corollary \ref{cr:XG}
that a toric divisor in $\XG$ is written as
$$
 D_{ab}^{P_m^{(n-3)}} = \{ Z_{ab}^{P_m^{(n-3)}} =0 \},
$$
where
$
 a, b \in  \{e_1, \dots, e_n, d_1, \dots, d_{n-3} \}
$
are a pair of edges in a triangle $P^{(n-3)}_m$
in the triangulation
$$
 P = P^{(n-3)}_1 \cup \dots \cup P^{(n-3)}_{n-2}
$$
of the reference polygon $P$.
The image of $D_{ab}^{P_m^{(n-3)}}$ in $\bP(\bigwedge^2 \bC^n)$
is given by $\bigcap_{i,j} \{Z_{ij}=0\}$, where the intersection is 
taken over all $i, j$ such that the path $\gamma(i,j)$
intersects with $a$ and $b$.
Then a torus-invariant subvariety $\bigcap D_{ab}^{P_m^{(n-3)}}$
is 0-dimensional exactly when there is a unique pair $(k,l)$
such that $\gamma(k,l)$ does not intersect 
$a, b \in \{e_1, \dots, e_n, d_1, \dots, d_{n-3} \}$ 
appearing the intersection.
\end{proof}

Suppose $|\bld{r}| = n$ so that the K\"ahler form on $\Gr(2,n)$ represents 
the first Chern class of $\Gr(2,n)$.
Then for each fixed point $p_{kl}$ in $\XG$, we have
\[
  \widetilde{\psi}_{e_i} (p_{kl}) = 
  \begin{cases}
    n & \text{if $i=k$ or $i=l$},\\
    0 & \text{otherwise}.
  \end{cases}
\]
and
\[
  \widetilde{\nu}_{\alpha} (p_{kl}) = 
  \begin{cases}
    n & \text{if $\{k, l \} \subset I_{\alpha}$}, \\
    0 & \text{otherwise}.
  \end{cases}
\]
This shows that
the vertices of the moment polytope
$\DeltaG$ of $\XG$
with respect to this symplectic form are lattice points,
so that $\DeltaG$ is an integral polytope.

\begin{df}
  A {\it reflexive polytope} $\Delta$ is an integral polytope such that
  \begin{itemize}
    \item $\Delta$ is given by 
        $\Delta = \{ u \in \bR^N \, | \, \langle v_i, u \rangle \ge -1 , \, 
        i=1, \dots, m \}$
      for some $v_1, \dots, v_m \in \mathbb{Z}^N$, where $m$ is the number of
      facets of $\Delta$, and
    \item $\Delta$ has the unique lattice point $0$ in its interior.
  \end{itemize}
\end{df}

\begin{prop}[Batyrev \cite{Batyrev_DPMS}]
The moment polytope of a polarized toric variety
is reflexive up to translation
if and only if it is a canonically-polarized
Gorenstein toric Fano variety.
\end{prop}

\begin{prop} \label{cor:Fano}
  If $|\bld{r}| = n$ then $\Delta_{\Gamma}$ is a 
  reflexive polytope up to translation by an integral vector.
  Hence $\XG$ is a Gorenstein toric Fano variety.
\end{prop}

\begin{proof}
Set $u'_{e_i} = u_{e_i} -2$ for each side $e_i$ and
$u'_{d_{\alpha}} = u_{d_{\alpha}} + 1 - |I_{\alpha}|$
for a diagonal $d_{\alpha} = \sum_{i \in I_{\alpha}} e_i$.
Then $(u'_a)_a = 0$ is equivalent to $u(a) = 1$ for all $a$,
where $u(a)$ are the coordinates defined by (\ref{eq:length_fn})
corresponding to lengths of sides and diagonals.
Recall that $\Delta_{\Gamma}$ is defined by triangle inequalities
$|u(a) - u(b)| \le u(c) \le u(a) + u(c)$
for each triangle in the triangulation $\Gamma$.
Then $(u'_a)_a = 0$ defines an interior point in $\Delta_{\Gamma}$.
We have to show that
\begin{itemize}
 \item
the triangle inequalities have the form
$\langle v, u' \rangle \ge -1$
for some integral vector $v \in \bZ^\N$, and
 \item
$(u'_a)_a = 0$ is the unique interior lattice point
in $\Delta_{\Gamma}$.
\end{itemize}
 We divide the proof into steps:

\begin{step}
The triangle inequality associated with a triangle
consisting of two edges $e_i, e_{i+1}$ and  a diagonal $d_{\alpha}$
has the form
$
 \langle v, u' \rangle \ge -1
$
for some integral vector $v \in \bZ^\N$.
\end{step}

In this case, one has $I_{\alpha} = \{i, i+1\}$ or 
$I_{\alpha} = \{1, \dots, n\} \setminus \{i, i+1\}$
depending on the orientation of $d_{\alpha}$.
The triangle inequalities in the first case are given by
\[
  \frac 12 |u'_{e_i} - u'_{e_{i+1}} | \le 
  \frac 12 (u'_{e_i} + u'_{e_{i+1}}) - u'_{\alpha,2} +1
  \le \frac 12 (u'_{e_i} + u'_{e_{i+1}}) +2,
\]
and these are equivalent to 
\begin{equation}
  u'_{d_{\alpha}} \ge -1, \quad
  u'_{e_i} - u'_{d_{\alpha}} \ge -1, \quad
  u'_{e_{i+1}} - u'_{d_{\alpha}} \ge -1.
  \label{eq:reflexive1}
\end{equation}
Similarly, the triangle inequalities for the second case are
\begin{equation}
  u'_{e_i} + u'_{e_{i+1}} + u'_{d_{\alpha}} \ge -1, \quad
  -u'_{e_i} - u'_{d_{\alpha}} \ge -1, \quad
  -u'_{e_{i+1}} - u'_{d_{\alpha}} \ge -1.
  \label{eq:reflexive1'}
\end{equation}
  
\begin{step}
The triangle inequality associated with a triangle
consisting of two diagonals 
$d_{\alpha}$, $d_{\beta}$, and a side $e_j$
has the form
$\langle v, u' \rangle \ge -1$
for some integral vector $v \in \bZ^\N$.
\end{step}

We may assume that $I_{\beta} = I_{\alpha} \cup \{ j\}$, or
$I_{\alpha} \cup I_{\beta} = \{1, \dots, n\} \setminus \{ j \}$ and 
$I_{\alpha} \cap I_{\beta} = \emptyset$, depending on the choice of 
orientations of the diagonals.
In the first case we have $d_{\beta} = d_{\alpha} + e_j$, and hence
the triangle inequalities are 
\[
  |u'_{d_{\alpha}} - u'_{d_{\beta}} + \frac 12 u'_{e_j} |
  \le \frac 12 u'_{e_j} + 1 \le 
  \sum_{i \in I_{\alpha}} u'_{e_i} + \frac 12 u'_{e_j}
  - u'_{d_{\alpha}} - u'_{d_{\beta}} + 2,
\]
which are equivalent to
\begin{equation}
  u'_{d_{\beta}} - u'_{d_{\alpha}} \ge -1, \quad
  u'_{e_j} + u'_{d_{\alpha}} - u'_{d_{\beta}} \ge -1, \quad
  \sum_{i \in I_{\alpha}} u'_{e_i} - u'_{d_{\alpha}} - u'_{d_{\beta}} \ge -1.
  \label{eq:reflexive2}
\end{equation}
The triangle inequalities for the second case are 
\begin{equation}
  \begin{split}
    -u'_{e_j} -u'_{d_{\alpha}} - u'_{d_{\beta}} &\ge -1,\\
    -u'_{d_{\alpha}} + u'_{d_{\beta}} + u'_{e_j} 
      + \sum_{i \in I_{\alpha}} u'_{e_i} &\ge -1, \\
    u'_{d_{\alpha}} - u'_{d_{\beta}}  
      - \sum_{i \in I_{\alpha}} u'_{e_i} &\ge -1.
  \end{split}
  \label{eq:reflexive2'}
\end{equation}

\begin{step}
The triangle inequality associated with a triangle
consisting of three diagonals 
$d_{\alpha}, d_{\beta}, d_{\gamma} = d_{\alpha} + d_{\beta}$
has the form
$
 \langle v, u' \rangle \ge -1
$
for some integral vector $v \in \bZ^\N$.
\end{step}

In this case,
we have 
\begin{align*}
  \sum_{i \in I_{\alpha}} u'_{e_i} 
    - u'_{d_{\alpha}} + u'_{d_{\beta}} + u'_{d_{\gamma}} &\ge -1,\\
  \sum_{i \in I_{\beta}} u'_{e_i} 
    + u'_{d_{\alpha}} + u'_{d_{\beta}} - u'_{d_{\gamma}} &\ge -1,\\
  - u'_{d_{\alpha}} - u'_{d_{\beta}} + u'_{d_{\gamma}} &\ge -1.
\end{align*}
If the orientations of $d_{\alpha}$, $d_{\beta}$, $d_{\gamma}$ 
are chosen in such a way that $d_{\alpha} + d_{\beta} + d_{\gamma} = 0$,
then $I_{\alpha} \cup I_{\beta} \cup I_{\gamma} = \{1,\dots, n\}$,
and hence the triangle inequalities are
\begin{align*}
  u'_{d_{\alpha}} - u'_{d_{\beta}} - u'_{d_{\gamma}} 
        - \sum_{i \in I_{\alpha}} u'_{e_i} &\ge -1,\\
  - u'_{d_{\alpha}} + u'_{d_{\beta}} - u'_{d_{\gamma}} 
        - \sum_{i \in I_{\beta}} u'_{e_i} &\ge -1,\\
  - u'_{d_{\alpha}} - u'_{d_{\beta}} + u'_{d_{\gamma}} 
        - \sum_{i \in I_{\gamma}} u'_{e_i} &\ge -1,
\end{align*}
as desired.

\begin{step}
$(u'_a)_a = 0$ is the unique interior lattice point 
in $\DeltaG$. 
\end{step}

Let $(u'_a)$ be an interior lattice point in $\DeltaG$.
For a triangle consisting of two edges $e_i, e_{i+1}$ and 
a diagonal $d_{\alpha}$, (\ref{eq:reflexive1}) 
or (\ref{eq:reflexive1'}) implies that $u'_{e_i}$, $u'_{e_{i+1}}$ and
$u'_{d_{\alpha}}$ satisfy
\[
  u'_{d_{\alpha}} \ge 0, \quad
  u'_{e_i} - u'_{d_{\alpha}} \ge 0, \quad
  u'_{e_{i+1}} - u'_{d_{\alpha}} \ge 0,
\]
or 
\[
  u'_{e_i} + u'_{e_{i+1}} + u'_{d_{\alpha}} \ge 0, \quad
  -u'_{e_i} - u'_{d_{\alpha}} \ge 0, \quad
  -u'_{e_{i+1}} - u'_{d_{\alpha}} \ge 0.
\]
Then we have $u'_{e_i} , u'_{e_{i+1}} \ge 0$ in either case.
Similarly, for a triangle consisting two diagonals $d_{\alpha}$,
$d_{\beta}$ and one side $e_j$,
we obtain $u'_{e_j} \ge 0$
from \eqref{eq:reflexive2}
or \eqref{eq:reflexive2'}.
Hence $u'_{e_i} \ge 0$  for all sides $e_1, \dots, e_n$.
Combining this with $\sum_{i=1}^n u'_{e_i} = 0$, 
we have $u'_{e_i}=0$, or equivalently, $u(e_i) = 1$ for all $e_1, \dots, e_n$.
Note that the coordinate change 
$(u'_{d_{\alpha}})_{\alpha} \mapsto (u(d_{\alpha}))_{\alpha}$ 
restricted to $u'_{e_1} = \dots = u'_{e_n} = 0$ is defined over 
$\mathbb{Z}$.
In particular, $(u'_{d_{\alpha}}) \in \mathbb{Z}^{n-3}$ if and only if
$(u(d_{\alpha}))_{\alpha} \in \mathbb{Z}^{n-3}$.
We take a triangle $P_1$ consisting of two sides $e_i, e_{i+1}$ and 
a diagonal $d_{\alpha}$.
The triangle inequalities
$0 = |u(e_i) - u(e_{i+1})| < u(d_{\alpha})
< u(e_i) + u(e_{i+1})=2$ implies that 
$u(d_{\alpha}) =1$, or equivalently $u'_{d_{\alpha}}=0$.
Then the $(n-1)$-gon $P \setminus P_1$ also has edges with unit lengths.
By repeating this process,
we obtain $u'_a=0$ for all $a$.
\end{proof}

Let
$$
 \YG = \Grtilde_{P^{(n-3)}_1} \times \dots \times \Grtilde_{P^{(n-3)}_{n-2}}
  \GIT_{(1, 1, \dots, 1)}  \bCx_0
   \times \bCx_{d_1} \times \dots \times \bCx_{d_{n-3}}
$$
be the symplectic reduction of
$
 \Grtilde_{P^{(n-3)}_1} \times \dots \times \Grtilde_{P^{(n-3)}_{n-2}}
$
at level $(1, 1, \dots, 1)$.
Since $\XG$ is the symplectic reduction
of the same space
at level $(1, 0, \dots, 0)$,
there there is a natural map $\pi : \YG \to \XG$.

\begin{prop} \label{prop:small_resol}
$\pi : \YG \to \XG$ is a small resolution.
\end{prop}

\begin{proof}
We first show that $\YG$ is smooth.
Recall that the moment map of the $S^1_{d_{\alpha}}$-action on 
$
 \Grtilde_{P^{(n-3)}_1} \times \dots \times \Grtilde_{P^{(n-3)}_{n-2}}
$
is given by
$$
 \mu_{S^1_{d_\alpha}} = \frac 12 \left(
  \left| Z_{a_1 d_{\alpha}}^{P_a} \right|^2
   + \left| Z_{a_2 d_{\alpha}}^{P_a} \right|^2
   - \left| Z_{b_1 d_{\alpha}}^{P_b} \right|^2
   - \left| Z_{b_2 d_{\alpha}}^{P_b} \right|^2
  \right).
$$
Then one can see that the $S^1_{d_{\alpha}}$-action
on $\mu_{S^1_{d_\alpha}}^{-1}(1)$ is free,
so that the symplectic reduction 
$$
 Y_{\Gamma} \cong 
  \mu_{S^1_0}^{-1}(1) \cap \mu_{S^1_{d_1}}^{-1} (1)
   \cap \dots \cap \mu_{S^1_{d_{n-3}}}^{-1} (1) /
   S^1_0 \times S^1_{d_1} \times \dots \times S^1_{d_{n-3}}
$$
is smooth.

The natural morphism
$\pi : \YG \to \XG$
sends a point $[y] \in \YG$ for
$$
 y \in
  \mu_{S^1_0}^{-1}(1) \cap \mu_{S^1_{d_1}}^{-1} (1)
   \cap \dots \cap \mu_{S^1_{d_{n-3}}}^{-1} (1)
$$
to $[x] \in \XG$
where $x$ is a point
$$
 x \in \overline{O_y} \cap \left(
 \mu_{S^1_0}^{-1}(1) \cap \mu_{S^1_{d_1}}^{-1} (0)
  \cap \dots \cap \mu_{S^1_{d_{n-3}}}^{-1} (0).
 \right)
$$
in the intersection of the closure of the
$\bCx_0 \times \bCx_{d_1} \times \dots \times \bCx_{d_{n-3}}$-orbit
$$
 O_y \subset 
  \Grtilde_{P^{(n-3)}_1} \times \dots \times \Grtilde_{P^{(n-3)}_{n-2}}
$$
of $y$.
One can see that
$$
 E_\Gamma = \{ [y] \in \YG \mid 
  Z_{b_1 d_{\alpha}}^{P_b}
  = Z_{b_2 d_{\alpha}}^{P_b}
  = 0 \text{ for some triangle $P_b$} \}.
$$
is the exceptional set of $\pi$;
the morphism $\pi$ is an isomorphism
outside $E_\Gamma$ since
$O_y$ is closed if $y \not \in E_\Gamma$,
and $\pi$ is not an isomorphism at $E_\Gamma$
since $\pi(y) \in \XG$ for $y \in E_\Gamma$ is singular
by Proposition \ref{pr:Sing_XG},
whereas $\YG$ is smooth everywhere.

Since $E \subset \YG$ is of codimension two,
the exceptional locus of $\pi$ does not contain a divisor,
and Proposition \ref{prop:small_resol} is proved.
\end{proof}

\begin{rem}
  Proposition \ref{cor:Fano} and Proposition \ref{prop:small_resol} are
  not true for toric degenerations of polygon spaces in general.
  Indeed, the triangulation $\Gamma_2$ in Example \ref{expl:pentagon} 
  gives a degeneration of the space $\pol$ of pentagons into the
  Hirzebruch surface $F_2$ of degree 2, which is not Fano.
  If we further assume that $r_1 = r_2$, then the central fiber is 
  the weighted projective plane $\bP (1,1,2)$, whose minimal resolution
  $F_2 \to \bP (1,1,2)$ is not small.
\end{rem}

\section{Potential functions}
 \label{sc:potential_function}

For a Lagrangian submanifold $L$ in a symplectic manifold,
the cohomology group $H^*(L; \Lambda_0)$
has a structure of a weak $A_\infty$-algebra
\cite{Fukaya-Oh-Ohta-Ono},
where  $\Lambda_0$ is the Novikov ring
$$
 \Lambda_0 = \left\{ \left.
  \sum_{i=0}^\infty a_i T^{\lambda_i}
  \right| a_i \in \bQ, \ \lambda_i \in \bR_{\ge 0}, \ 
   \lim_{i \to \infty} \lambda_i = \infty \right\}.
$$
A solution to the Maurer-Cartan equation
$$
 \sum_{k=0}^\infty \frakm_k(b, \dots, b) \equiv 0
  \mod \PD([L])
$$
is called a weak bounding cochain,
where $\PD ([L])$ is the Poincar\'e dual
of the fundamental class $[L]$.
The potential function is a map
$
 \po : \scM(L) \to \Lambda_0
$
from the moduli space $\scM(L)$ of weak bounding cochains 
defined by
$$
 \sum_{k=0}^\infty \frakm_k(b, \dots, b) = \po(b) \cdot \PD([L]).
$$
Cho and Oh \cite{Cho-Oh}
and
Fukaya, Oh, Ohta and Ono \cite{FOOO_toric_I}
computed the potential functions for Lagrangian torus orbits 
in toric manifolds.
This is generalized
in \cite{Nishinou-Nohara-Ueda_TDGCSPF,
Nishinou-Nohara-Ueda_PFTD}
to an integrable system on a Fano manifold
which has a degeneration
into the toric moment map on a toric Fano variety
admitting a small resolution.
From Corollary \ref{cor:Fano} and Proposition \ref{prop:small_resol}, 
we can apply this result
to the toric degeneration of the integrable system $\Psi_\Gamma$
on $\Gr(2,n)$
to obtain the following:

\begin{thm} \label{th:potential}
  Fix a triangulation $\Gamma$ of the reference polygon and
  let $\ell_i(u) = \langle v_i, u \rangle - \tau_i$
  be the affine functions defining $\Delta_{\Gamma}:$
  \[
    \Delta_{\Gamma}
    = \{ u \in \bR^\N \mid \ell_i(u) \ge 0, \ i = 1, \dots, m \}.
  \]
  Then for any $u \in \Int \Delta_{\Gamma}$, one has an inclusion
  $H^1(L(u); \Lambda_0) \subset \mathcal{M}(L(u))$
  for the Lagrangian torus fiber $L(u) = \Psi_{\Gamma}^{-1}(u)$,
  and the potential function is given by
  \[
    \po_{\Gamma} (L(u), x) 
    = \sum_{i=1}^m e^{\langle v_i, x \rangle} T^{\ell_i(u)}
  \]
  for $x \in H^1(L(u), \Lambda_0) \cong \Lambda_0^\N$.
\end{thm}

By setting $y_a = e^{x_a}T^{u_a}$ and $Q = T^{|\bld{r}|}$, 
the potential function can be regarded as a Laurent polynomial
in $y_a$ and $Q$.
Since $\Delta_{\Gamma}$ is given by triangle inequalities
\begin{align*}
  - u(a) + u(b) + u(c) &\ge 0, \\
    u(a) - u(b) + u(c) &\ge 0, \\
    u(a) + u(b) - u(c) &\ge 0
\end{align*}
in terms of the ``length coordinates'' defined in (\ref{eq:length_fn}),
the potential function can be written as
\[
  \po_{\Gamma} = \sum_{\text{triangles}} 
        \left( 
          \frac{y(b) y(c)}{y(a)}
        + \frac{y(a) y(c)}{y(b)}
        + \frac{y(a) y(b)}{y(c)}
        \right),
\]
where $y(a)$ is a Laurent monomial in $y_{e_i}^{1/2}$,
$y_{d_{\alpha}}$ and $Q$ defined by
\[
  y(a) = \begin{cases}
    y_{e_i}^{1/2}, & a = e_i \quad (i=1, \dots, n-1),\\
    Q (y_{e_1} \dots y_{e_{n-1}})^{-1/2}, & a = e_n, \\
    y_{d_{\alpha}}^{-1} \prod_{i \in I_{\alpha}} y_{e_i}^{1/2},
    & a = d_{\alpha},
  \end{cases}
\]
and the sum is taken over all triangles 
in the triangulation $\Gamma$.
This is the potential function
given in \eqref{eq:potential_function}, and
Theorem \ref{th:potential_function} is proved.

Now we consider the relation between the potential functions 
$\po_{\Gamma_1}$ and $\po_{\Gamma_2}$ corresponding to 
two different triangulations $\Gamma_1$ and $\Gamma_2$.
It suffices to consider the case where $\Gamma_1$ is transformed into $\Gamma_2$
by a single Whitehead move replacing a diagonal $d$ with $d'$.
Recall that the piecewise linear transform  (\ref{eq:PL_transf}) sends 
$\Delta_{\Gamma_1}$ to $\Delta_{\Gamma_2}$.
We define its {\it geometric lift} in the sense of
\cite{Berenstein-Zelevinsky_TPM} by
\begin{equation}
  y(d')  = y(d) \cdot
   \frac{y(a_1) y(a_4) + y(a_2) y(a_3)}
        {y(a_1) y(a_2) + y(a_3) y(a_4)},
   \label{eq:geom_lift}
\end{equation}
which means that the tropicalization of this map gives 
the piecewise-linear transformation (\ref{eq:PL_transf}).
From (\ref{eq:PL_transf}) or a direct computation, 
(\ref{eq:geom_lift}) is also written as
\begin{equation}
   y(d') = y(d) \cdot 
   \frac {\ \displaystyle{\frac{y(a_1)}{y(a_2)} + \frac{y(a_2)}{y(a_1)} 
     + \frac{y(a_3)}{y(a_4)} + \frac{y(a_4)}{y(a_3)} \ }}
   {\displaystyle{ \frac{y(a_1)}{y(a_4)} + \frac{y(a_4)}{y(a_1)} 
     + \frac{y(a_2)}{y(a_3)} + \frac{y(a_3)}{y(a_2)} }}.
  \label{eq:geom_lift2}
\end{equation}
Then Theorem \ref{th:potential} gives the following:

\begin{cor} \label{cr:geometric_lift}
  Under the above situation, 
  the potential functions $\po_{\Gamma_1}$ and $\po_{\Gamma_2}$ 
  are related
  by the rational map \eqref{eq:geom_lift}.
\end{cor}

\begin{proof}
The potential function corresponding to $\Gamma_1$ is written as
\[
  \begin{split}
    \po_{\Gamma_1} = y(d) \left( 
     \frac{y(a_1)}{y(a_2)} + \frac{y(a_2)}{y(a_1)} 
     + \frac{y(a_3)}{y(a_4)} + \frac{y(a_4)}{y(a_3)}
     \right) \\ + 
     \frac {y(a_1) y(a_2) + y(a_3) y(a_4)}{y(d)} + F(y),
   \end{split}
\]
where $F(y)$ is a Laurent polynomial which does not contain $y_d$.
Since the triangle inequalities for $\Gamma_1$ and $\Gamma_2$ are the same 
except for those containing $d$ and $d'$,
the potential function for $\Gamma_2$ is written as
\[
  \begin{split}
    \po_{\Gamma_2} = y(d') \left(
      \frac{y(a_1)}{y(a_4)} + \frac{y(a_4)}{y(a_1)} 
      + \frac{y(a_2)}{y(a_3)} + \frac{y(a_3)}{y(a_2)}
      \right)  \\ +
      \frac{y(a_1) y(a_4) + y(a_2) y(a_3)}{y(d')} + F(y).
    \end{split}
\]
Hence the coordinate change \eqref{eq:geom_lift}
transforms $\po_{\Gamma_1}$ into $\po_{\Gamma_2}$.
\end{proof}

Theorem \ref{th:geometric_lift} is a direct consequence
of Corollary \ref{cr:geometric_lift}.


\begin{rem}
In the case of flag manifolds,
Rusinko \cite{Rusinko}
proved a similar result for string polytopes.
\end{rem}

\begin{expl}
  Consider a triangulation $\Gamma$ 
  of a quadrilateral given by $d= e_1 + e_2$.
  The triangle inequalities for $\Gamma$ are
  \begin{align*}
    \langle (1,0,0,-1), (u_{e_1}, u_{e_2}, u_{e_3}, u_{d}) \rangle
      \phantom{ \,\,\, + | \bld{r} | } &\ge 0,\\
    \langle (0,1,0,-1), (u_{e_1}, u_{e_2}, u_{e_3}, u_{d}) \rangle
      \phantom{ \,\,\, + | \bld{r} | } &\ge 0,\\
    \langle (0,0,0,1), (u_{e_1}, u_{e_2}, u_{e_3}, u_{d}) \rangle
      \phantom{ \,\,\, + | \bld{r} | } &\ge 0,\\
    \langle (0,0,-1,-1), (u_{e_1}, u_{e_2}, u_{e_3}, u_{d}) \rangle
      + | \bld{r} | &\ge 0,\\
    \langle (-1,-1,0,1), (u_{e_1}, u_{e_2}, u_{e_3}, u_{d}) \rangle
      + | \bld{r} | &\ge 0,\\
    \langle (1,1,1,-1), (u_{e_1}, u_{e_2}, u_{e_3}, u_{d}) \rangle
      - | \bld{r} | &\ge 0.
  \end{align*}
  Thus the potential function is
  \[
    \po_{\Gamma} 
    = \frac{y_{e_1}}{y_{d}} + \frac{y_{e_2}}{y_{d}} + y_{d}
     + \frac{Q}{y_{e_3} y_{d}} + \frac{Q y_{d}}{y_{e_1} y_{e_2}}
     + \frac{y_{e_1} y_{e_2} y_{e_3}}{Q y_{d}} .
  \]
  After the coordinate change
  \[
    y^{(1)}_1 = y_{e_1}, \quad
    y^{(2)}_1 = \frac{y_{e_1} y_{e_2}}{y_d}, \quad
    y^{(2)}_2 = y_{d}, \quad
    y^{(3)}_2 = \frac{y_{e_1}y_{e_2}y_{e_3}}Q,
  \]
  which is a geometric lift of (\ref{eq:coord_GC}), 
  the potential function becomes
  \begin{equation}
    \po_{\Gamma} = y^{(2)}_2 + \frac{y^{(1)}_1}{y^{(2)}_2} 
    + \frac{y^{(3)}_2}{y^{(2)}_2} + \frac{y^{(2)}_1}{y^{(1)}_1}
    + \frac{y^{(2)}_1}{y^{(3)}_2} + \frac{Q}{y^{(2)}_1}.
  \label{eq:potential_GC}
  \end{equation}
\end{expl}

\begin{rem}
The potential function in \eqref{eq:potential_GC} coincides
with the superpotential 
in \cite[(B.2)]{Eguchi-Hori-Xiong_GQC}.
More generally, for every $n$, the potential function
corresponding to the caterpillar
coincides with \cite[(B.25)]{Eguchi-Hori-Xiong_GQC}
after a coordinate change 
corresponding to \eqref{eq:coord_GC}.
\end{rem}


\bibliographystyle{amsalpha}
\bibliography{bibs}

\def\cprime{$'$} \def\cprime{$'$}
\providecommand{\bysame}{\leavevmode\hbox to3em{\hrulefill}\thinspace}
\providecommand{\MR}{\relax\ifhmode\unskip\space\fi MR }
\providecommand{\MRhref}[2]{%
  \href{http://www.ams.org/mathscinet-getitem?mr=#1}{#2}
}
\providecommand{\href}[2]{#2}
\begin{thebibliography}{FOOO10}

\bibitem[Aur07]{Auroux_MSTD}
Denis Auroux, \emph{Mirror symmetry and {$T$}-duality in the complement of an
  anticanonical divisor}, J. G\"okova Geom. Topol. GGT \textbf{1} (2007),
  51--91. \MR{2386535 (2009f:53141)}

\bibitem[Aur09]{Auroux_SLFWCMS}
\bysame, \emph{Special {L}agrangian fibrations, wall-crossing, and mirror
  symmetry}, Surveys in differential geometry. {V}ol. {XIII}. {G}eometry,
  analysis, and algebraic geometry: forty years of the {J}ournal of
  {D}ifferential {G}eometry, Surv. Differ. Geom., vol.~13, Int. Press,
  Somerville, MA, 2009, pp.~1--47. \MR{MR2537081}

\bibitem[Bat94]{Batyrev_DPMS}
Victor~V. Batyrev, \emph{Dual polyhedra and mirror symmetry for {C}alabi-{Y}au
  hypersurfaces in toric varieties}, J. Algebraic Geom. \textbf{3} (1994),
  no.~3, 493--535. \MR{MR1269718 (95c:14046)}

\bibitem[BZ01]{Berenstein-Zelevinsky_TPM}
Arkady Berenstein and Andrei Zelevinsky, \emph{Tensor product multiplicities,
  canonical bases and totally positive varieties}, Invent. Math. \textbf{143}
  (2001), no.~1, 77--128. \MR{MR1802793 (2002c:17005)}

\bibitem[CO06]{Cho-Oh}
Cheol-Hyun Cho and Yong-Geun Oh, \emph{Floer cohomology and disc instantons of
  {L}agrangian torus fibers in {F}ano toric manifolds}, Asian J. Math.
  \textbf{10} (2006), no.~4, 773--814. \MR{MR2282365 (2007k:53150)}

\bibitem[EHX97]{Eguchi-Hori-Xiong_GQC}
Tohru Eguchi, Kentaro Hori, and Chuan-Sheng Xiong, \emph{Gravitational quantum
  cohomology}, Internat. J. Modern Phys. A \textbf{12} (1997), no.~9,
  1743--1782. \MR{MR1439892 (99a:32027)}

\bibitem[FOOOa]{FOOO_LFTMS}
Kenji Fukaya, Yong-Geun Oh, Hiroshi Ohta, and Kaoru Ono, \emph{{L}agrangian
  {F}loer theory and mirror symmetry on compact toric manifolds},
  arXiv:1009.1648.

\bibitem[FOOOb]{FOOO_toric_deg}
\bysame, \emph{Toric degeneration and non-displaceable {L}agrangian tori in
  {$S^2 \times S^2$}}, arXiv:1002.1660.

\bibitem[FOOO09]{Fukaya-Oh-Ohta-Ono}
\bysame, \emph{Lagrangian intersection {F}loer theory: anomaly and
  obstruction}, AMS/IP Studies in Advanced Mathematics, vol.~46, American
  Mathematical Society, Providence, RI, 2009. \MR{MR2553465}

\bibitem[FOOO10]{FOOO_toric_I}
\bysame, \emph{Lagrangian {F}loer theory on compact toric manifolds. {I}}, Duke
  Math. J. \textbf{151} (2010), no.~1, 23--174. \MR{2573826}

\bibitem[Fot00]{Foth_MSP}
Philip Foth, \emph{Moduli spaces of polygons and punctured {R}iemann spheres},
  Canad. Math. Bull. \textbf{43} (2000), no.~2, 162--173. \MR{1754021
  (2001e:14012)}

\bibitem[GM82]{Gelfand-MacPherson_GGGD}
I.~M. Gel{\cprime}fand and R.~D. MacPherson, \emph{Geometry in {G}rassmannians
  and a generalization of the dilogarithm}, Adv. in Math. \textbf{44} (1982),
  no.~3, 279--312. \MR{658730 (84b:57014)}

\bibitem[GS83]{Guillemin-Sternberg_GCS}
V.~Guillemin and S.~Sternberg, \emph{The {G}el\cprime fand-{C}etlin system and
  quantization of the complex flag manifolds}, J. Funct. Anal. \textbf{52}
  (1983), no.~1, 106--128. \MR{MR705993 (85e:58069)}

\bibitem[HK97]{Hausmann-Knutson_PSG}
Jean-Claude Hausmann and Allen Knutson, \emph{Polygon spaces and
  {G}rassmannians}, Enseign. Math. (2) \textbf{43} (1997), no.~1-2, 173--198.
  \MR{1460127 (98e:58035)}

\bibitem[HMM]{Howard-Manon-Millson}
B.~{Howard}, C.~{Manon}, and J.~{Millson}, \emph{{The Toric Geometry of
  Triangulated Polygons in Euclidean Space}}, arXiv:0810.1352.

\bibitem[Jef94]{Jeffrey_EMS}
Lisa~C. Jeffrey, \emph{Extended moduli spaces of flat connections on {R}iemann
  surfaces}, Math. Ann. \textbf{298} (1994), no.~4, 667--692. \MR{1268599
  (95g:58030)}

\bibitem[Kly94]{Klyachko_SP}
Alexander~A. Klyachko, \emph{Spatial polygons and stable configurations of
  points in the projective line}, Algebraic geometry and its applications
  ({Y}aroslavl\cprime, 1992), Aspects Math., E25, Vieweg, Braunschweig, 1994,
  pp.~67--84. \MR{1282021 (95k:14015)}

\bibitem[KM96]{Kapovich-Millson_SGPES}
Michael Kapovich and John~J. Millson, \emph{The symplectic geometry of polygons
  in {E}uclidean space}, J. Differential Geom. \textbf{44} (1996), no.~3,
  479--513. \MR{1431002 (98a:58027)}

\bibitem[Kos66]{Kostant_OSSRT}
Bertram Kostant, \emph{Orbits, symplectic structures and representation
  theory}, Proc. {U}.{S}.-{J}apan {S}eminar in {D}ifferential {G}eometry
  ({K}yoto, 1965), Nippon Hyoronsha, Tokyo, 1966, p.~p. 71. \MR{0213476 (35
  \#4340)}

\bibitem[KY02]{Kamiyama-Yoshida}
Yasuhiko Kamiyama and Takahiko Yoshida, \emph{Symplectic toric space associated
  to triangle inequalities}, Geom. Dedicata \textbf{93} (2002), 25--36.
  \MR{1934683 (2003g:53149)}

\bibitem[MP01]{Millson-Poritz}
John~J. Millson and Jonathan~A. Poritz, \emph{Around polygons in {$\Bbb R\sp
  3$} and {$S\sp 3$}}, Comm. Math. Phys. \textbf{218} (2001), no.~2, 315--331.
  \MR{1828984 (2002b:53132)}

\bibitem[NNU]{Nishinou-Nohara-Ueda_PFTD}
Takeo Nishinou, Yuichi Nohara, and Kazushi Ueda, \emph{Potential functions via
  toric degenerations}, arXiv:0812.0066.

\bibitem[NNU10]{Nishinou-Nohara-Ueda_TDGCSPF}
\bysame, \emph{Toric degenerations of {G}elfand-{C}etlin systems and potential
  functions}, Adv. Math. \textbf{224} (2010), no.~2, 648--706. \MR{2609019}

\bibitem[Rua01]{Ruan_I}
Wei-Dong Ruan, \emph{Lagrangian torus fibration of quintic hypersurfaces. {I}.
  {F}ermat quintic case}, Winter School on Mirror Symmetry, Vector Bundles and
  Lagrangian Submanifolds (Cambridge, MA, 1999), AMS/IP Stud. Adv. Math.,
  vol.~23, Amer. Math. Soc., Providence, RI, 2001, pp.~297--332. \MR{MR1876075
  (2002m:32041)}

\bibitem[Rua02]{Ruan_II}
\bysame, \emph{Lagrangian torus fibration of quintic {C}alabi-{Y}au
  hypersurfaces. {II}. {T}echnical results on gradient flow construction}, J.
  Symplectic Geom. \textbf{1} (2002), no.~3, 435--521. \MR{1959057
  (2004b:32040)}

\bibitem[Rus08]{Rusinko}
J.~Rusinko, \emph{Equivalence of mirror families constructed from toric
  degenerations of flag varieties}, Transform. Groups \textbf{13} (2008),
  no.~1, 173--194. \MR{MR2421321}

\bibitem[SS04]{Speyer-Sturmfels_TG}
David Speyer and Bernd Sturmfels, \emph{The tropical {G}rassmannian}, Adv.
  Geom. \textbf{4} (2004), no.~3, 389--411. \MR{2071813 (2005d:14089)}

\bibitem[Tre02]{Treloar}
Thomas Treloar, \emph{The symplectic geometry of polygons in the 3-sphere},
  Canad. J. Math. \textbf{54} (2002), no.~1, 30--54. \MR{1880958 (2002k:53168)}

\end{thebibliography}

\ \vspace{0mm} \\

\noindent
Yuichi Nohara \\
Department of Mathematics, 
School of Science and Technology, \\
Meiji University\\
1-1-1 Higashi-Mita, Tama-ku, Kawasaki-shi, 
Kanagawa 214-8571 , Japan\\
{\em e-mail address}\ : \  nohara@meiji.ac.jp
\ \vspace{0mm} \\

\noindent
Kazushi Ueda \\
Graduate School of Mathematical Sciences, 
The University of Tokyo,\\
3-8-1 Komaba,
Meguro-ku,
Tokyo,
153-8914,
Japan\\
{\em e-mail address}\ : \  kazushi@ms.u-tokyo.ac.jp
\ \vspace{0mm} \\

\end{document}